\documentclass[12pt]{article}

\usepackage[utf8]{inputenc}
\usepackage[english]{babel}
\usepackage[T1]{fontenc}
\usepackage{epigraph}
\usepackage{enumitem}
\usepackage{amsmath,latexsym,amsthm,mathtools}
\usepackage{amssymb}
\usepackage[bookmarks,bookmarksnumbered,pdfstartview=FitBH,backref=page]{hyperref}
\hypersetup{
	colorlinks=true, 
	linkcolor=blue,
	citecolor= blue,
	linktoc= all, 
}

\numberwithin{equation}{section}

\newtheorem{proposition}[equation]{Proposition}
\newtheorem{theorem}[equation]{Theorem}
\newtheorem{corollary}[equation]{Corollary}
\newtheorem{lemma}[equation]{Lemma}

\newtheorem{claim}[equation]{Claim}

\newtheorem{theoremintro}{Theorem}

\newtheorem{propositionintro}[theoremintro]{Proposition}

\theoremstyle{definition}
\newtheorem{definition}[equation]{Definition}

\newtheorem{remark}[equation]{Remark}
\newtheorem*{remark*}{Remark}

\newtheorem*{question*}{Question}
\newtheorem{questionintro}{Question}

\newenvironment{cproof}{\begin{proof}[Proof of the
        claim]}{\end{proof}}
\newenvironment{cproofbis}[1]{\begin{proof}[#1]}{\end{proof}}

\def\defin#1{\textbf{#1}} 
\DeclareMathOperator{\dom}{\mathrm{dom}}
\DeclareMathOperator{\rng}{\mathrm{rng}}
\newcommand{\vv}{v}

\newcommand{\N}{\mathbb{N}}
\newcommand{\Z}{\mathbb{Z}}

\newcommand{\inv}{^{-1}}
\newcommand{\la}{\left\langle}
\newcommand{\ra}{\right\rangle}
\newcommand{\lala}{\left\langle\!\left\langle}
\newcommand{\rara}{\right\rangle\!\right\rangle}

\newcommand{\abs}[1]{\left\lvert #1\right\rvert}

\newcommand{\Ac}{\mathcal{A}}
\newcommand{\Fc}{\mathcal{F}}
\newcommand{\Gc}{\mathcal{G}}

\newcommand{\Hc}{\mathcal{H}}
\newcommand{\Ic}{\mathcal{I}}

\newcommand{\Nc}{\mathcal{N}}
\newcommand{\Oc}{\mathcal{O}}

\newcommand{\Rc}{\mathcal{R}}

\newcommand{\Vc}{\mathcal{V}}

\newcommand{\PK}{\mathcal{K}}
\newcommand{\QQ}{\mathcal{Q}}

\def\id{\mathrm{id}}

\newcommand\BSo{\mathrm{BS}}
\newcommand\BSe{\mathbf{BS}}
\newcommand\Phe{\mathrm{Ph}}
\DeclareMathOperator\PHE{\mathbf{Ph}}
\newcommand\PHEred{\mathbf{Ph}_{\textrm{red}}}
\newcommand\primes{\mathcal{P}}

\newcommand\degout{\deg_{\mathrm{out}}}
\newcommand\degin{\deg_{\mathrm{in}}}
\newcommand{\Sym}{\mathrm{Sym}}
\newcommand{\Stab}{\mathrm{Stab}}
\newcommand{\Sub}{\mathrm{Sub}}
\newcommand{\Sch}{\mathbf{Sch}}
\newcommand\Schreier{\Sch}
\newcommand\Cayley{\mathbf{Cay}}

\newcommand{\trrule}{\mathtt{r}}

\newcommand\Tree{\mathcal{T}}

\newcommand{\source}{\mathtt{s}}
\newcommand{\target}{\mathtt{t}}

\newcommand{\bs}{\backslash}

\usepackage{xcolor}

\usepackage{tikz}
\usetikzlibrary{arrows}
\usetikzlibrary{cd}

\usepackage{enumitem}
\setlist{nosep}

\newcommand\pathinT{c}

\author{
	Damien Gaboriau, \\
	François Le Maître and Yves Stalder
}

\title{On the space of subgroups of Baumslag-Solitar groups II: High transitivity}
\date{\today}

\begin{document}
\maketitle
\begin{abstract}
    We continue our study of the perfect kernel of the space of transitive actions of Baumslag-Solitar
    groups by investigating high transitivity.
    We show that actions of finite phenotype are never highly transitive,
    except when the phenotype is $1$, in which case high transitivity is actually generic. 
    In infinite phenotype, high transitivity is generic, except 
    when $\lvert m \rvert= \lvert n\rvert$ 
    where it is empty.
    We also reinforce the dynamical properties of the action by conjugation on the perfect kernel that we had established in our first paper, replacing topological transitivity by high topological transitivity.
\end{abstract}
	
	\smallskip
	
	{
		\small	
		\noindent\textbf{{MSC-classification:}}	
		20B22; 37B; 20E06; 20E08; 20F65.
	}

\tableofcontents
\section{Introduction}

In order to study the possible dynamics of a given countable group \(\Gamma\), a natural first 
step is to study its transitive actions.
One particularly striking property a transitive action can have is \textbf{high transitivity} (see Definition~\ref{def: HT}). 
A basic example is the natural action of the group $\Sym_f(\N)$ of finitely supported permutations of \(\N\). 
The first example of a finitely generated group with a highly transitive faithful action was exhibited by B. H. Neumann who pointed out that the natural action of 
$\Sym_f(\Z)\rtimes\Z$ on $\Z$ is also highly transitive. Since the latter group is two-generated,
it follows that the free group $\mathbb F_2$ admits a (non-faithful!) highly transitive action. 
As for faithful highly transitive actions of the free group, they were initially constructed by McDonough
\cite{mcdonoughPermutationRepresentationFree1977}.
This strong property was eventually proved to be generic among all transitive $\mathbb F_2$-actions \cite{dixon_most_1990}. 

Subsequently, many groups were shown to admit faithful highly transitive actions, see e.g. 
\cite{chaynikov_properties_2012, kitroserHighlytransitiveActionsSurface2012, garionHighlyTransitiveActions2013, moonHighlyTransitiveActions2013, fima_highly_2015,
hullTransitivitydegreescountable2016, fima_characterization_HT_2022}
and references therein. The proofs often rely on Baire category arguments in a tailored space of actions
(see Remark \ref{rmk: ht for hull osin} for details on how to recast Chaynikov's and Hull-Osin's inductive constructions using the Baire category theorem). 
A notable exception is \cite{garionHighlyTransitiveActions2013},
where explicit faithful highly transitive actions of  $\operatorname{Out}(\mathbb F_n)$ are exhibited.

Since transitive actions are completely encoded by the conjugacy class of 
the associated stabilizer subgroup, one is naturally led to the study of the \(\Gamma\)-action by conjugation on its  
space of subgroups \(\Sub(\Gamma)\). This space carries a natural compact Polish topology. 
High transitivity only makes sense for actions on infinite sets, 
so we focus on the Polish subspace \(\Sub_{[\infty]}(\Gamma)\) of 
infinite index subgroups.
Denote by \(\mathcal{HT}(\Gamma)\) the set of subgroups \(\Lambda\in \Sub_{[\infty]}(\Gamma)\) such that \(\Lambda\bs\Gamma\curvearrowleft\Gamma\) is highly transitive. 
This set is \(G_\delta\) (see Lemma~\ref{lem: Sub infty and HT are Gdelta}).

\begin{questionintro}\label{qu: HT dense}
When is it true that
	\(\mathcal{HT}(\Gamma)\) is dense in \(\Sub_{[\infty]}(\Gamma)\)?
\end{questionintro}

We are thus asking what is the class of countable groups whose {\em generic}
transitive actions on an infinite set are actually highly transitive.

It sometimes happens that 
the \( \Gamma \)-action
on \(\Sub_{[\infty]}(\Gamma)\) is topologically transitive \cite{azuelosPerfectKernelDynamics2023}.
In this case, the topological zero-one law \cite[Theorem~8.46]{kechris_classical_1995} ensures that either the $G_\delta$ set
\(\mathcal{HT}(\Gamma)\)
is dense in \(\Sub_{[\infty]}(\Gamma)\), or its complement contains 
a dense \(G_\delta\) set. 
For instance, when \(\Gamma\) is the group \(\Sym_f(\N)\), 
even though \(\Gamma\) acts topologically transitively on \(\Sub_{[\infty]}(\Gamma)\) (see for instance \cite[Ex.~9.43 and Prop.~9.44]{le_maitre_polish_2024}), 
\(\mathcal{HT}(\Gamma)\) consists of a single (meager) conjugacy class 
by \cite[Prop.~2.4]{leboudecConfinedSubgroupsHigh2022}.

A first example with dense \(\mathcal{HT}(\Gamma)\) is provided by free groups as a consequence of Dixon's aforementioned result.
More generally, this holds for free products \(\Gamma_1*\Gamma_2\), 
with \(\abs{\Gamma_1}\geq 2\) and \(\abs{\Gamma_2}\geq 3\) by \cite[Thm.~10.30]{le_maitre_polish_2024}.
A natural next step is to consider groups acting on trees.
In the aforementioned work \cite{fima_characterization_HT_2022}, it was shown that when 
$\Gamma$ admits a faithful minimal action of general type on a tree,  then
$\Gamma$ admits a faithful highly transitive action as soon as the action on the boundary of \(\Tree\)
is topologically free. A key example is given by Baumslag-Solitar groups.
For these groups, we previously showed that the \( \Gamma \)-action  on \(\Sub_{[\infty]}(\Gamma)\) is not topologically transitive.

To be more precise, let us fix some parameters $m,n\in\Z$ with \(\abs{m}\geq 2\) and \(\abs{n}\geq 2\) and consider the Baumslag-Solitar group 
\[\Gamma=\BSo(m,n) \coloneqq \la b,t\mid t b^m t^{-1}=b^n\ra.\]
In \cite{CGLMS-22} we unveiled a natural partition of \(\Sub(\Gamma)\) 
into \(\Gamma\)-invariant subsets provided by a
conjugation invariant map
\(\PHE_{m,n}:\Sub(\Gamma)\to \Z_{\geq 1}\cup\{\infty\}\) 
that we call the  \emph{$(m,n)$-phenotype map}.

This phenotype $\PHE_{m,n}(\Lambda)$ is obtained from the index $[\la b\ra : \la b\ra \cap \Lambda]$ as follows: 
 if $[\la b\ra : \la b\ra \cap \Lambda]=\infty$, then $\PHE_{m,n}(\Lambda)=\infty$. Otherwise, remove from the prime factors decomposition of $[\la b\ra : \la b\ra \cap \Lambda]$
 all the prime numbers $p$ that appear in $m$ or $n$, 
 except those such that $\abs{m}_p=\abs{n}_p<\abs{[\la b\ra : \la b\ra \cap \Lambda]}_p$, 
 where $\abs{k}_p$ denotes the $p$-adic valuation of $k$
 (see Section \ref{sec:phenotype} for details).
 Thus $\PHE_{m,n}(\la b \ra)=1$ and $ \PHE_{m,n}(\{\id\})=\infty$ and 
 $\PHE_{m,n}(\la b^q \ra)=q$ for every $q\in\Z_{\geq 1}$ relatively prime to both \(m\) and \(n\). 
 In particular, the set of possible phenotypes
 \(\QQ_{m,n}\coloneqq \PHE_{m,n}(\Sub(\Gamma))\) is 
 infinite and always contains $1$ and $\infty$, independently of $(m,n)$.
 Observe that $\PHE_{m,n}\inv(\infty)$ contains only infinite index subgroups.

Next, we proved that the perfect kernel $\PK(\Gamma)$ 
of the space $\Sub(\BSo(m,n))$ is the 
set of subgroups $\Lambda\leq\Gamma$
such that the double quotient  $\Lambda\bs\Gamma/\la b\ra$ is infinite.
Letting \(\PK_q\coloneqq \PHE_{m,n}\inv(q)\cap \PK(\Gamma)\) for all phenotype \(q\in\mathcal Q_{m,n}\),
we obtain a \(\Gamma\)-invariant decomposition 
\[
\PK(\Gamma)= \bigsqcup_{q\in\mathcal Q_{m,n}} \PK_q.
\]
All the pieces \(\PK_q\) for \(q<\infty\) are open in \(\PK(\Gamma)\),
while \(\PK_\infty\) is closed\footnote{Furthermore, when \(\abs m=\abs n\),
all the pieces \(\PK_q\) for \(q<\infty\) are clopen.}.
A key result from \cite{CGLMS-22} is the fact that the \(\BSo(m,n)\)-action
on each of the pieces \(\PK_q\) of the above partition is topologically 
transitive.

Furthermore, the perfect kernel is almost equal to the space of infinite index subgroups:
if we define the remaining piece as 
\[\mathcal{C}_{\infty}\coloneqq\Sub_{[\infty]}(\Gamma)\smallsetminus \PK(\Gamma),\] 
then 
\(\mathcal C_\infty\) is a countable open subset of \(\PHE_{m,n}\inv(\infty)\) which is empty
exactly when \(\abs m\neq \abs n\). 
The complete picture for the space of infinite index subgroups 
is thus as follows:
\[
\Sub_{[\infty]}(\BSo(m,n))= \underbrace{\mathcal{C}_{\infty}\sqcup {\PK_{\infty}}}_{\Sub_{[\infty]}\cap \PHE\inv(\infty)}\
\phantom{aaaaaaaaaaa}\llap{$\overbrace{\phantom{aaaa}
\sqcup \bigsqcup_{q\in\mathcal Q_{m,n}\smallsetminus\{\infty\}} \PK_q}^{\PK(\BSo(m,n))}$}
\]

The main goal of this article is to highlight a striking phenomenon: 
the subgroups giving rise to highly transitive actions are located in 
certain very specific pieces $\PK_q$ of the perfect kernel 
(depending on the parameters $m$ and $n$). 
On the contrary, the subgroups of the other pieces never even 
correspond to primitive actions.
The precise statement is as follows:
\begin{theoremintro}
\label{th-intro: HT generic in Ph= 1 and oo}
    Let \( m,n\in\Z \) such that \( \abs{m}\geq 2 \) and \( \abs{n}\geq 2 \), 
	let  \( \Gamma=\BSo(m,n) \). 
	Let \( q \) be an \( (m,n) \)-phenotype and let 
    \[
	\mathcal{HT}_q\coloneqq \{\Lambda\in\PK_q \colon 
					\Lambda\bs\Gamma\curvearrowleft \Gamma\text{ is highly transitive}\}.
	\]
    Then \( \mathcal{HT}_q\) is dense \( G_\delta \) in \( \PK_q  \) exactly when either
    \begin{itemize}
        \item  \( q=1 \),  or
        \item  \( q=\infty \) and  \( \abs{m} \neq \abs{n} \).
    \end{itemize}
Otherwise, \( \mathcal{HT}_q \) is empty, 
and infinite index subgroups of phenotype \(q\)
actually never even correspond to primitive actions.  
\end{theoremintro}

Since it is more enjoyable, in connection with high transitivity, to consider faithful actions, 
we use the results of \cite{CGLMS-22} to localize the subgroups that give rise to faithful actions.

\begin{propositionintro}[Proposition~\ref{prop: faithfulness according to phenotype}]
\label{prop-intro: faithfulness according to phenotype}
The set of faithful actions forms a dense \( G_\delta \) set inside $\PK_q$ for every phenotype $q$ when $\abs{m} \neq \abs{n}$, or for $q=\infty$ and $\abs{m} = \abs{n}$; otherwise, $\PK_q$ contains no faithful action.

Moreover, if $\Lambda$ belongs to $\mathcal{C}_\infty$, 
then the action $\Lambda\bs \BSo(m,n)\curvearrowleft \BSo(m,n)$ is not faithful.
\end{propositionintro}

Theorem \ref{th-intro: HT generic in Ph= 1 and oo} is
proved in Section \ref{Sec: HT}.
We summarize the results of Theorem~\ref{th-intro: HT generic in Ph= 1 and oo} and Proposition~\ref{prop-intro: faithfulness according to phenotype} about high transitivity and faithfulness in $\PK_q$ in the following tables. 
\begin{center}
\begin{tabular}{ll}
Case $\abs{m} \neq \abs{n}$
&
\begin{tabular}{|c||c|c|}\hline
 \emph{Phenotypes}& \emph{High transitivity} & \emph{Faithfulness} \\
\hline\hline
  $q=1$ & dense \( G_\delta \) & dense \( G_\delta \) \\
  \hline
  $q=\infty$ & dense \( G_\delta \) & dense \( G_\delta \)\\
  \hline
  \emph{other phenotypes} & $\emptyset$  & dense \( G_\delta \)\\
  \hline
  \end{tabular}
\\
\\
 Case $\abs{m} = \abs{n} $ 
 &
\begin{tabular}{|c||c|c|}\hline
 \emph{Phenotypes}& \emph{High transitivity} & \emph{Faithfulness} \\
\hline\hline
  $q=1$ & dense \( G_\delta \) & $\emptyset$ \\
  \hline
  $q=\infty$ & $\emptyset$  & dense \( G_\delta \)\\
  \hline
  \emph{other phenotypes} & $\emptyset$ &  \( \emptyset \)\\
  \hline
  \end{tabular}
  \end{tabular}
\end{center}  
\medskip
These tables can be completed by mentioning that
if one concentrates on 
$\Sub_{[\infty]}(\BSo(m,n))$ instead of $\PK(\BSo(m,n))$,
we should add in the case $\abs{m}=\abs{n}$
a line for $\mathcal C_{\infty}$,
whose elements are never highly transitive nor faithful.

High transitivity admits a topological version which is an enrichment of topological transitivity.
Namely, an action $X\curvearrowleft \Gamma$ on an infinite perfect Polish space $X$,
is called \defin{highly topologically transitive} when the diagonal action $X^d\curvearrowleft \Gamma$ is topologically transitive for every \(d\geq 1\). 

We have established in \cite{CGLMS-22} the topological transitivity of the action of $\BSo(m,n)$ on each \(\PK_q\). We now upgrade this to high topological transitivity.

\begin{theoremintro}
	\label{th-intro: htt}
	Let \( m,n \) be integers such that \({\abs m}, \abs n\geq 2 \). 
	Then for every phenotype \(q\in \QQ_{m,n}\) 
	the action by conjugation of \( \BSo(m,n)  \) on the invariant subspace 
	\( \PK_q(\BSo(m,n))  \) is highly topologically transitive.
\end{theoremintro}

Theorem \ref{th-intro: HT generic in Ph= 1 and oo} was
implicitly proved in \cite[Theorem~4.4]{fima_characterization_HT_2022}
in the case $\PK_{\infty}(\BSo(m,n))$ for \(\abs{m}\neq\abs{n}\).
We adapt these techniques and develop a unified framework (see in particular Lemma \ref{lem: leaving the seed uniformly})
towards establishing high transitivity
or high topological transitivity.
This framework is based on pre-actions and their maximal forest saturations, as in our first paper \cite{CGLMS-22}.
A notable difference is that we often have to stay at the level of pre-actions instead of 
working with the more flexible \((m,n)\)-graphs.
We clarify the definition of maximal forest saturations by showing that 
each pre-action admits a \emph{unique} maximal forest saturation
(this result is a special case of Theorem \ref{thm: forest saturation ruled}).
We also identify precisely the corresponding stabilizer subgroup
(see Corollary \ref{cor: maximal forest sat stabilizer}).

We would also like to highlight the work of Sasha Bontemps \cite{Bontemps-pkGBS}, who extended 
the main results of \cite{CGLMS-22} to Generalized Baumslag-Solitar groups.
She discovered the right notion of phenotype in this wider context and obtained a description analogous to ours of the perfect kernel.
Her phenotype yields a decomposition of the perfect kernel into invariant pieces where she shows the action is highly topologically transitive, thus obtaining a natural generalization of
Theorem~\ref{th-intro: htt}.

Let us finally mention a result of independent interest that relates high transitivity to its topological counterpart:
\begin{theoremintro}[HT + TT $\implies$ HTT, 
Theorem~\ref{thm: HT + TT implies HTT}]\label{th-intro: ht + tt impl htt}
For any countable group $\Gamma$, 
if  $\mathcal P\subseteq \Sub(\Gamma)$ is a $\Gamma$-invariant 
$G_\delta$ subset such that 
\begin{enumerate}[label=(\arabic*)] 
	\item $\mathcal{HT}(\Gamma) \cap \mathcal P$ is dense in $\mathcal P$, and
	\item the $\Gamma$-action on $\mathcal P$ is topologically transitive,
\end{enumerate}
then  the $\Gamma$-action on $\mathcal P$ is highly topologically transitive.
\end{theoremintro}

\paragraph{Acknowledgements.} 
    The authors are deeply grateful to Alessandro
    Carderi for more mathematical conversations
    than they can count. We will miss him in our future projects and we wish him the greatest success in his new adventures.
    We are thankful to Sasha Bontemps for useful remarks on a preliminary version
    of our work.
   
    The authors also acknowledge funding by the LABEX MILYON (ANR-10-LABX-0070) of Université de Lyon, within the program “Investissements d’Avenir” (ANR-11-IDEX-0007) operated by the French National Research Agency (ANR).
	D.~G.\ is supported by the CNRS.
	D.~G.\ and F.~L.M.\ acknowledge funding by the ANR project PLAGE (ANR-24-CE40-3137).

\section{General results on high transitivity and its topological counterpart}

This section collects some results on high transitivity
which are not specific to Baumslag-Solitar groups. 
Most of them are well-known, but Theorem \ref{thm: HT + TT implies HTT}
appears to be new. 
We also propose a slight modification of the usual definition of high topological transitivity 
so that it encompasses high transitivity (see Lemma \ref{lem: HTT vs HTT} for the comparison
with the definition used in \cite{azuelosPerfectKernelDynamics2023}).

\subsection{Preliminaries on high (topological) transitivity}

We begin by recalling the definition of high transitivity.
To this end, recall that given a right action \( \alpha \) of a group \( \Gamma \)
on a set \(X\) with at least \(d\in\N\) elements, we have a natural diagonal action \(\alpha^{(d)}\)
on the set \(X^{(d)}\subseteq X^d\) 
made of \(d\)-tuples consisting of pairwise distinct elements of \(X\).
We say then say that \( \alpha \) is \(d\)\defin{-transitive} when \(\alpha^{(d)}\) is transitive. 

\begin{definition}\label{def: HT}
	A right action \( \alpha \) of a group \( \Gamma \) on an infinite set \(X\) is called
	\defin{highly transitive (HT)} when it is \(d\)-transitive for every \(d\in\N\), namely:
	for every \(x_1,\dots,x_d\in X\) pairwise distinct and 
	\(y_1,\dots,y_d\in X\) pairwise distinct, there is \(\gamma\in\Gamma\) such that for 
	all \( i\in\{1,\dots,d\} \), we have \(x_i\alpha(\gamma)=y_i\).
\end{definition}

We will make use of the following two well-known lemmas. The first says that we may as well 
assume \(x_i\neq y_j\) for all \(i,j\in\{1,\dots,d\}\) in the above definition.
\begin{lemma}[{\cite[Lem.~2.2]{fima_characterization_HT_2022}}]
	\label{lem: ht with all distinct}
	Let \( \alpha \) be a right \( \Gamma \)-action on an infinite set \(X\),
	then the action is highly transitive iff for all \(d\in\N\) 
	and all \(x_1,\dots,x_{2d}\in X\) pairwise distinct, 
	there is \( \gamma\in\Gamma \) such that 
	\[
	\forall i\in\{1,\dots,d\}, \quad x_i \alpha(\gamma)=x_{i+d}.
	\]
\end{lemma}

\begin{lemma}[{\cite[Lem.~1.4(2)]{moonHighlyTransitiveActions2013}}]
	\label{lem: ht passes to normal subgroups}
	Let \( \alpha \) be a highly transitive \( \Gamma \)-action on an infinite set \(X\).
	If $N$ is a normal subgroup of \( \Gamma \), then the restriction of \( \alpha \) to \(N\) is either 
	trivial or highly transitive.
\end{lemma}

We now introduce the topological versions of these notions. 
We choose a definition which slightly differs from the classical one
(they are equivalent when the underlying space is infinite Hausdorff perfect, see Lemma~\ref{lem: HTT vs HTT}).
The advantage of our choice is that an action on an infinite set \(X\)
is highly transitive if and only if it is highly topologically transitive when endowing
\(X\) with the discrete topology. 
In general,  
    	\(X^{(d)}\) is endowed with the topology induced form the product topology on \(X^{d}\).
Note that when \(X\)
is infinite compact, then \(X^{(d)}\) is not compact anymore.

\begin{definition}
    \label{def: htt action}
	    	    Let $\alpha$ be a  right action of a group \( \Gamma \) on a Hausdorff topological space \(X\)
    	by homeomorphisms.

    	\begin{enumerate}
    	\item 
The action is called \defin{topologically transitive~(TT)}
    	if whenever \(U, V\) are nonempty open subsets of \(X\), 
    	there is \(\gamma\in\Gamma\) such that \[ U\alpha(\gamma)\cap V\neq \emptyset.\]
    	
    	\item Given \(d\in\N\), if \(X\) contains at least \(d\) elements, the action is furthermore called \defin{\(d\)-topologically transitive (\(d\)-TT)} when the diagonal action 
    	\(\alpha^{(d)}\) on \(X^{(d)}\) is topologically transitive.
    	    
	    \item If \(X\) is infinite, \( \alpha \) is called \defin{highly topologically transitive (HTT)}
    	when it is \(d\)-topologically transitive for every \(d\in\N\). 
    	\end{enumerate}
\end{definition}

Note that in the definition of topological transitivity, 
we may always restrict to \(U\) and \(V\) belonging to a basis of the topology of \(X\).
The next lemma connects this definition with the more classical definition, where 
\(X^{(d)}\) is replaced by \(X^d\). 
Recall that a topological space is called \defin{perfect}
when it has no isolated points. 

\begin{lemma}\label{lem: HTT vs HTT}
	Let \( \alpha \) be a \( \Gamma \)-action by homeomorphisms on a 
	Hausdorff infinite topological space \(X\). Then \( \alpha \) is highly topologically transitive
	in the sense of Definition \ref{def: htt action}
	iff one of the following holds:
	\begin{itemize}
		\item \(X\) is perfect and for every \( d\in\N \), the diagonal action \(\alpha^d\) on 
				\(X^d\) is topologically transitive;
		\item The set of isolated points of \(X\) is dense and the restriction of 
				\( \alpha \) to this set is highly transitive (in particular, this set consists of a single orbit).
	\end{itemize}
\end{lemma}
\begin{proof}
	First observe that since \(X\) is Hausdorff, 
	the sets \(U_1\times\cdots\times U_d \),
	where \(U_i\subseteq X\) is open and the \(U_i\)'s are pairwise disjoint,
	form a basis for the topology of \(X^{(d)}\).
	It follows that \( \alpha \) is highly topologically transitive
	iff for all \( d\in\N \), for all \(U_1,\dots,U_d\subseteq X\) pairwise disjoint, open, and nonempty,
	and for all \(V_1,\dots,V_d\subseteq X\) pairwise disjoint, open, and nonempty,
	there is \( \gamma\in\Gamma \) such that for all \( i\in \{1,\dots,d\} \), 
	\( U_i\alpha(\gamma)\cap V_i\neq \emptyset \).
	Let us distinguish two cases.

	\paragraph{Case 1: \( X \) is perfect.}
	Assuming first that \( \alpha \) is highly topologically transitive,
	we now show that the action \(\alpha^d\) on \(X^d\) is topologically transitive. 
	To this end, by the above paragraph
	it suffices to show that given any nonempty open sets \(U_1,\dots,U_d\subseteq X\),
	there are \emph{pairwise disjoint} nonempty open sets \(U'_1\subseteq U_1,\dots,U'_d\subseteq U_d\). 
	Since \(X\) is perfect, every \(U_i\) is infinite, so we can select pairwise distinct
	\( x_i\in U_i\), and then since \(X\) is Hausdorff each \(x_i\) admits a neighborhood
	\(U'_i\subseteq U_i\) such that \({(U'_i)}_{i=1}^d\) consists of pairwise disjoint open sets, as 
	required. 

	Conversely, it is obvious that, for every \( d \), topological transitivity of 
	\( \alpha^d \) implies that \(\alpha^{(d)}\) is also topologically transitive.
	Conversely, it is obvious that, for every \( d \), topological transitivity of 
	\( \alpha^d \) implies that \(\alpha^{(d)}\) is also topologically transitive.

	\paragraph{Case 2: \( X \) is not perfect.}
	Let us first assume that \( \alpha \) is highly topologically transitive. By definition, \(X\) contains at least one isolated 
	point \(x_0\), and topological transitivity ensures us that for every non empty open set \(U\),
	there is some \( \gamma\in\Gamma \) such that \( \{ x_0 \}\alpha(\gamma)\cap U\neq\emptyset \), 
	i.e.\ the isolated point \(x_0\alpha(\gamma)\) belongs to \(U\). 
	So the set of isolated points is dense. 
	Using singletons as open sets, it is then straightforward to see that the restriction of \( \alpha \) 
	to the set of isolated points is highly transitive as desired.

	Conversely, assume that the set of isolated points is dense and that the 
	restriction of \( \alpha \) to this set is highly transitive. 
	Let \(U_1,\dots,U_d\) and \(V_1,\dots,V_d\) be pairwise disjoint nonempty open subsets of \(X\), 
	for every \( i\in \{1,\dots,d\} \)
	we pick \(x_i\in U_i\) isolated and \(y_i\in V_i\) isolated, then the \(x_i\)'s and the \(y_i\)'s are 
	pairwise distinct by disjointness, so by assumption we find \( \gamma\in\Gamma \) taking each
	\(x_i\) to \(y_i\), in particular \( U_i\alpha(\gamma) \cap V_i\neq \emptyset \) as desired.
\end{proof}

\begin{remark}\label{rmk: disjoint version for htt}
	Similar considerations yield that Lemma~\ref{lem: ht with all distinct} can be generalized as follows:
	a \( \Gamma \)-action \( \alpha \) by homeomorphisms on an infinite Hausdorff topological space \(X\) is highly topologically transitive if and only if 
	for every \( d\in\N \), and every \defin{pairwise disjoint} nonempty open sets \( U_1,\dots,U_{2d} \),
	there is \( \gamma\in\Gamma \) such that for all \( i\in \{1,\dots,d\} \), 
	we have \( U_i\alpha(\gamma)\cap U_{i+d}\neq\emptyset \). 

	On the contrary, the natural generalization of Lemma~\ref{lem: ht passes to normal subgroups} fails badly:
	high topological transitivity does not pass to normal subgroups acting faithfully. Here is a counterexample:
	let \( \Gamma=\Z\times\Z/2\Z \), and consider the natural \( \Gamma \)-action on \( (\Z/2\Z)^{\Z} \) where 
	\( \Z \) acts by shift and \( \Z/2\Z \) acts diagonally by translation. 
	The \( \Z \)-action is  highly topologically transitive, in particular the \(\Gamma\)-action
	is highly topologically transitive. 
	However, the normal subgroup \( \Z/2\Z \) acts faithfully,
	but its action is not even topologically transitive. 
\end{remark}

\begin{remark}
    Letting \( X^{(\infty)}=\{(x_n)\in X^\N\colon\forall i\neq j, x_i\neq x_j \} \), 
    one can conveniently reformulate high topological transitivity as the fact that the 
    diagonal action \(\alpha^{(\infty)}\) on \( X^{(\infty)} \) is topologically
    transitive. 
    If $X$ is discrete, then $\alpha$ is highly transitive if and only if \(\alpha^{(\infty)}\) is topologically     transitive.
\end{remark}

\subsection{Space of subgroups and high transitivity}

Recall that \(\Sub(\Gamma)\) is the set of subgroups of the countable group $\Gamma$. It is a Polish space when equipped with the topology defined by a basis of open (in fact clopen) subsets given by \(\Vc(\Ic,\Oc)=\{\Lambda\leq \Gamma\mid \Ic\subseteq \Lambda \text{ and } \Oc\cap \Lambda=\emptyset\}\)
where \(\Ic,\Oc\) are finite subsets of $\Gamma$.

   As in~\cite[Section~2.2]{CGLMS-22}, we freely identify the compact Polish space \(\Sub(\Gamma)\) 
with the space of isomorphism classes of transitive pointed
right \( \Gamma \)-actions. Such isomorphism classes are usually denoted \([\alpha,x_0]\).
Recall that \( \Gamma \) acts on \(\Sub(\Gamma)\) by conjugacy, and that in terms 
of isomorphism classes of transitive pointed right actions, this action is given by moving the base point: 
\[[\alpha,x_0]\cdot \gamma \coloneqq [\alpha,x_0\alpha(\gamma)].\]

We denote by \(\Sub_{[\infty]}(\Gamma)\subseteq\Sub(\Gamma)\) the set of infinite index subgroups of \( \Gamma \) and by 
\(\mathcal{HT}(\Gamma)\subseteq \Sub_{[\infty]}(\Gamma)\) the set of infinite index subgroups $\Lambda$ such that $\Lambda\bs \Gamma \curvearrowleft \Gamma$ is highly transitive.

Recall that a subset of a topological space is $G_\delta$ 
if it can be written as a countable intersection
of open subsets. 
Every closed subset of a Polish space is \(G_\delta\).
Furthermore, a subset of a Polish space is Polish for the induced topology
if and only if it is \( G_\delta \) \cite[Theorem~3.11]{kechris_classical_1995}. 

\begin{lemma}
\label{lem: Sub infty and HT are Gdelta}
The set  \(\Sub_{[\infty]}(\Gamma)\) is a $G_\delta$ subset of \(\Sub(\Gamma)\)
which is closed when $\Gamma$ is finitely generated.
The set \(\mathcal{HT}(\Gamma)\) is a \( G_\delta \) subset of \(\Sub(\Gamma)\) 
(and of \(\Sub_{[\infty]}(\Gamma)\)).
In particular, \(\mathcal{HT}(\Gamma)\) and  \(\Sub_{[\infty]}(\Gamma)\) are both 
Polish spaces for their induced topology.
\end{lemma}

We will merge the proof of Lemma~\ref{lem: Sub infty and HT are Gdelta} together with the proof of the following application of the Baire Category Theorem.

\begin{lemma}\label{lem: HT and Baire}
	Let \( \Gamma \) be a countable group. 
	Let \(\mathcal P \subseteq \Sub_{[\infty]}(\Gamma)\) be a \( G_\delta \) (hence Polish) subspace.
	Then the set \(\mathcal{HT}_{\mathcal P}\coloneqq \mathcal{HT}(\Gamma)\cap \mathcal P\) is dense in \(\mathcal P\)
	if and only if the following holds:
	\begin{enumerate}
	\item[(\textasteriskcentered)]%
	For every \([\alpha,x]\in \mathcal P \),
	 for every \( d\in\N \), for every \( g_1,\dots,g_{2d}\in \Gamma \) such that \({(x\alpha(g_i))}_{i=1}^{2d}\)
	consists of pairwise distinct elements, and for every neighborhood \(\mathcal V\)
	of \([\alpha,x]\), there is \([\alpha',x']\in\mathcal P\cap \mathcal V\) and 
	\( \gamma\in \Gamma \) such that for all \( i\in \{1,\dots,d\} \), we have
	\[
	x'\alpha'(g_i\gamma)=x'\alpha'(g_{i+d}).
	\]
\end{enumerate}

\end{lemma}
\begin{proof}[Proof of Lemmas~\ref{lem: Sub infty and HT are Gdelta} and \ref{lem: HT and Baire}]
For every $n$, the set of subgroups of index at least $n$
is the open set of those $\Lambda$ for which there exist pairwise distinct $g_1, \ldots g_n \in \Gamma$
such that $g_i g_j\inv \notin \Lambda$ for all $i\neq j$. Thus 
\(\Sub_{[\infty]}(\Gamma)=\bigcap_{n\in \N}\bigcup_{(g_i)_i\in \Gamma^{(n)}}
\mathcal V(\emptyset, g_i g_j\inv )\) is \( G_\delta \).
When $\Gamma$ is finitely generated, its finite index subgroups are isolated, thus $\Sub_{[\infty]}(\Gamma)$ is closed.

	We now show that the space \(\mathcal{HT}(\Gamma)\) is \( G_\delta \) 
	in \(\Sub_{[\infty]}(\Gamma)\), thus also showing that \(\mathcal{HT}_{\mathcal P}=\mathcal{HT}(\Gamma)\cap \mathcal P\)
	is \( G_\delta \) in \(\mathcal P\). 
	First observe that, by Lemma~\ref{lem: ht with all distinct}, a pointed transitive action \( [\alpha,x] \) on an infinite set is highly transitive
	iff for all \( d\in\N \), and for all \( g_1,\dots,g_{2d}\in\Gamma \) such that 
	\( x\alpha(g_i)\neq x\alpha(g_j) \) for all distinct \( i,j\in \{1,\dots,2d\} \), 
	there is \( \gamma\in\Gamma \) such that
	\( x\alpha(g_i\gamma)=x\alpha(g_{i+d})\) for all \( i\in \{1,\dots,d\}\).
	For every \( g_1,\dots,g_{2d}\in\Gamma \), let us thus denote by \( \mathcal U_{g_1,\dots,g_{2d}} \) the set of 
	all infinite index subgroups \( \Lambda\leq \Gamma \) such that at least one of the following holds:
	\begin{itemize}
		\item there is some \( i\neq j\in \{1,\dots,2d\} \) such that 
			  \( g_i g_j\inv\in\Lambda \);
		\item there is \( \gamma\in\Gamma \) such that 
		\(g_i\gamma g_{i+d}\inv\in\Lambda\), for all \(\forall i\in \{1,\dots,d\}\).
	\end{itemize}
	It is clear from their definition that all \( \mathcal U_{g_1,\dots,g_{2d}} \) are open. Our first observation combined with the correspondence between pointed actions 
	and subgroups yield that
	\[
	\mathcal{HT}(\Gamma)=\bigcap_{d\geq 1}\bigcap_{g_1\dots g_{2d}\in \Gamma} \mathcal U_{g_1,\dots,g_{2d}}.
	\]
	It follows that \(\mathcal{HT}(\Gamma)\) is \( G_\delta \) as wanted.

	Now if \(\mathcal{HT}_{\mathcal P}\) is dense in \(\mathcal P\),  Condition~(\textasteriskcentered) of Lemma~\ref{lem: HT and Baire}
	clearly holds. Conversely, assume that  Condition~(\textasteriskcentered) of Lemma~\ref{lem: HT and Baire} holds. 
	By the Baire category theorem applied to the Polish space \(\mathcal P\),
	it suffices to show that
	each intersection \(\mathcal P \cap\mathcal U_{g_1,\dots,g_{2d}}\) is dense in \(\mathcal P\).

	Let us thus fix some subgroup \(\Lambda\in\mathcal P\) and \( g_1,\dots,g_{2d}\in\Gamma \). We need to show that \( \Lambda \) is a limit of elements of \(\mathcal P\cap \mathcal U_{g_1,\dots,g_{2d}}\).
	If for all \(i\neq j\), we have \( g_i g_j\inv\not\in\Lambda \),  Condition~(\textasteriskcentered) of Lemma~\ref{lem: HT and Baire} precisely yields that conclusion.
	Otherwise, we have some \(i\neq j\) such that \( g_i g_j\inv\in\Lambda \), but then 
	\( \Lambda \) is already in \(\mathcal U_{g_1,\dots,g_{2d}}\), so we are done as well.
\end{proof}

Observe that if  Condition~(\textasteriskcentered) of Lemma~\ref{lem: HT and Baire} holds
for an arbitrary subset \(\mathcal P\subseteq \Sub_{[\infty]}(\Gamma)\)
then it also holds for its closure in \(\Sub_{[\infty]}(\Gamma)\).
This immediately yields the following result.

\begin{corollary}\label{cor: HT via dense set}
	Let \( \Gamma \) be a countable group. 
	Let \(\mathcal P \subseteq \Sub_{[\infty]}(\Gamma)\) be a subset satisfying  Condition~(\textasteriskcentered) of Lemma~\ref{lem: HT and Baire}.
	Let \(\overline{\mathcal P}\) denote the closure of \(\mathcal P\) in
	\( \Sub_{[\infty]}(\Gamma) \).
	Then the set \(\mathcal{HT}_{\overline{\mathcal P}}\) of highly transitive pointed actions 
	in \(\overline{\mathcal P}\) is dense \( G_\delta \) in \(\overline{\mathcal P}\).\qed{}
\end{corollary}

\begin{remark}\label{rmk: ht for hull osin}
	Corollary \ref{cor: HT via dense set} can be used to show Hull and Osin's result on 
	high transitivity a bit differently (compare with~\cite[Sec.~3.3]{hullTransitivitydegreescountable2016}). 
	Recall that their result is that every acylindrically hyperbolic group
	with trivial finite radical is highly transitive.
	For this, they crucially use the notion of \emph{small subgroups}
	of such a group (see~\cite[Definition~2.10]{hullTransitivitydegreescountable2016}),
	which always has infinite index. 
	Denote by \( \mathcal P \) the set of all small subgroups of 
	a fixed acylindrically hyperbolic group \( \Gamma \) with trivial
	finite radical.  
	The fact that \( \Gamma \) is highly transitive
	can then be deduced directly from the following two statements,
	denoting by \(\overline{\mathcal P}\) the closure of \(\mathcal P\) 
	in \(\Sub_{[\infty]}(\Gamma)\):
	\begin{itemize}
		\item  By~\cite[Prop.~3.1]{hullTransitivitydegreescountable2016},
		    the set \( \mathcal P \) of small subgroups satisfies Condition~(\textasteriskcentered) of Lemma~\ref{lem: HT and Baire}.
			So we can apply Corollary \ref{cor: HT via dense set} and obtain that \(\mathcal{HT}_{\overline{\mathcal P}}\) 
			is dense \( G_\delta \) in \( \overline{\mathcal{P}} \)
			(note that their result is slightly stronger than  Condition~(\textasteriskcentered) of Lemma~\ref{lem: HT and Baire}).
		\item By~\cite[Lem.~2.11]{hullTransitivitydegreescountable2016}, every 
		element of \(\mathcal P\) is faithful. Since the set of faithful actions 
		is \( G_\delta \), we conclude that the set of faithful actions
		is dense \(G_\delta \) in \(\overline{\mathcal P}\).
	\end{itemize}
	By the Baire category theorem and the two above items, the set of highly transitive faithful actions
	is dense \( G_\delta \) in \(\overline{\mathcal P}\), in particular it is not empty and hence
	\( \Gamma \) is highly transitive. 
	A similar approach works for Chaynikov's 
	result on high transitivity for hyperbolic groups \cite{chaynikov_properties_2012}, 
	replacing 
	small subgroups by quasi-convex subgroups.
\end{remark}

We now observe a natural connection between high transitivity and high topological transitivity.
This statement will not be used in the proof of Theorem~\ref{th-intro: htt}, so the hasty reader can safely skip it, but it initially led us to some
cases which we state as Corollary~\ref{cor: htt from ht and tt for BS}.

\begin{theorem}[HT + TT \(\Rightarrow\) HTT]
	\label{thm: HT + TT implies HTT}
	Let \( \Gamma \) be a countable group and let \(\mathcal P\subseteq \Sub(\Gamma)\) be a \( \Gamma \)-invariant 
	nonempty Polish subset. 
	Assume that the set of highly transitive actions \([\alpha, x_0] \in \mathcal P\) 
	is dense in \(\mathcal P\) and that 
	the \( \Gamma \)-action on \(\mathcal P\) is topologically transitive.
	Then the \( \Gamma \)-action on \(\mathcal P\) is highly topologically transitive.
\end{theorem}
\begin{proof}
	Following Remark~\ref{rmk: disjoint version for htt}, we fix \(U_1,\dots,U_{2d}\) non empty pairwise disjoint open subsets of 
	\(\mathcal P\) and need to find \( \gamma\in\Gamma \) such that \(U_j\cdot \gamma \cap U_{j+d}\neq\emptyset \)
	for all \(j\in \{1,\dots,d\} \). 
	By topological transitivity, letting \( g_1=\id \), we can inductively find 
	\( g_2,\dots,g_{2d}\in\Gamma \)
	so that the open set \( U\coloneqq\bigcap_{i=1}^{2d}  U_i\cdot g_i\inv \) is not empty. 

	By density, we may and do fix a highly transitive action 
	\([\alpha, x_0] \in  U\). Then by the definition of \(U\),
	for all \( i\in \{1,\dots,2d\} \)
	we have \([\alpha, x_0\alpha(g_i)]\in U_i\). 
	Since \(U_1,\dots,U_{2d}\) are pairwise disjoint, the 
	elements \(x_0\alpha(g_1),\dots,x_0\alpha(g_{2d})\) are pairwise distinct.

	By high transitivity, there is then \( \gamma\in\Gamma \) such that for 
	all \( j\in \{1,\dots,d\} \), we have \(x_0\alpha(g_j)\alpha(\gamma)=x_0\alpha(g_{j+d}) \),
	hence \[[\alpha, x_0\alpha(g_j)] \cdot \gamma = [\alpha, x_0\alpha(g_{j+d})] \quad \text{ in } \Sub(\Gamma).\]
	Using that \([\alpha, x_0\alpha(g_i)]\in U_i\) for all \( i\in \{1,\ldots,2d\} \),
	we conclude that \( U_j\cdot\gamma \cap U_{j+d}\neq\emptyset \) 
	for all \( j\in \{1,\dots,d\} \) as wanted.
\end{proof}

As an application to Baumslag-Solitar groups, we obtain the following particular case of Theorem~\ref{th-intro: htt} from pre-existing results.
\begin{corollary}\label{cor: htt from ht and tt for BS}
	Let \( \Gamma = \BSo(m,n) \) where \( \abs{m}, \abs{n} \geq 2 \), \(\abs{m}\neq\abs{n}\),
	and let \( \PK_\infty(\Gamma) \) be the set of infinite phenotype subgroups in the perfect kernel.
	Then, the action \( \PK_\infty(\Gamma) \curvearrowleft \Gamma \) is highly topologically transitive.
\end{corollary}

\begin{proof}
	Let us fix \( x_0 \in \N \).
	Let \( \Ac \) denote the set of transitive \( \Gamma \)-actions \( \alpha \) on \( \N \)
	such that \([\alpha, x_0] \in \PK_\infty(\Gamma)\).
	This is exactly the set of transitive \( \Gamma \)-actions \( \alpha \) on \( \N \)
	such that the permutation \(\alpha(b)\) has infinitely many infinite orbits and no finite orbit (in particular, \( \Ac \) does not depend on the choice of \( x_0 \)).

	We endow \( \Ac \) with the pointwise convergence topology of \( \operatorname{Hom}(\Gamma, \Sym(\N)) \). 
	For every \(\sigma\in \Sym(\N)\) with infinitely many infinite orbits and no finite orbit, 
	the set of \( \alpha\in\Ac \) such that \(\alpha(b)=\sigma\) is a homeomorphic copy of the space \(\mathbf{TA}\) of~\cite[Definition~4.2]{fima_characterization_HT_2022}; 
	and these copies form a partition of \( \Ac \).
	Thus, Theorem 4.4 in the same article shows that highly transitive actions form a dense subset in \( \Ac \)
	(in fact a dense \( G_\delta \) subset in each copy of \(\mathbf{TA}\)).

	Since the map \(\Ac \to \PK_\infty(\Gamma); \alpha \mapsto [\alpha, x_0]\) is continuous and surjective,
	we obtain that
	highly transitive pointed actions form a dense subset in \(\PK_\infty(\Gamma)\).
	Finally, the action \( \PK_\infty(\Gamma) \curvearrowleft \Gamma \) is topologically transitive by~\cite[Theorem~B]{CGLMS-22}.
	Hence, Theorem~\ref{thm: HT + TT implies HTT} applies.
\end{proof}

\section{Background around Baumslag-Solitar groups}


\subsection{Graphs} 

We keep the same definitions and conventions about graphs as in~\cite{CGLMS-22}. 
Let us nevertheless emphasize the following points concerning graphs.
\begin{itemize}
	\item Paths are sequences of edges, and by convention paths of length \( 0 \) are vertices.
	\item Given a vertex \( v \), the ball of radius \( R\in\N \) is the graph induced by
	the set of vertices reachable from \(v\) by a path of length \( \leq R\), denoted by \(B(v,R)\).
	In particular, it may contain edges which connect two vertices at distance \(R\) from the identity, but which do not belong to any path of length \(\leq R\) starting from \( v \). 

	\item Given an edge \(e\), the half-graph of \(e\) is the connected component of 
	\(\target(e)\) obtained when removing the edge \(e\) from the ambient graph.
	When this connected component is a tree, we rather say the half-graph is a 
	\textbf{half-tree}.
	This is typically the case when the ambient graph is a tree, 
	but we will mostly use this 
	terminology in the forest saturation, which will be reviewed in the next section.
\end{itemize}

\subsection{Pre-actions} 

We continue our preliminaries by recalling what pre-actions are, motivated by the 
presentation of Baumslag-Solitar groups: 
\[
\BSo(m,n) \coloneqq \la b,t\mid t b^m t^{-1}=b^n\ra.
\]
Now, a (right) \defin{preaction} of \( \BSo(m,n) \) on a set \( X \) is a couple \( (\beta,\tau) \)
where \( \beta \) is a bijection of \(X\) and \( \tau \) is a \emph{partial} bijection
 of \(X\) such  that  \(\dom(\tau)\) is \(\beta^n\)-invariant,
\( \rng(\tau) \) is \( \beta^m \)-invariant and for all \(x\in\dom\tau\) we have
\( x \tau \beta^m = x\beta^n \tau \).
 The set \( X \) is called the \textbf{domain} of the pre-action.
 We also sometimes write pre-actions as triples \( (X,\beta,\tau) \)
 in order to make the domain explicit.
 We say that a pre-action is \textbf{saturated} when \(\dom\tau=\rng\tau=X\).
 Note that saturated pre-actions are the exact same thing as \(\BSo(m,n)\)-actions.

An \textbf{extension} of a pre-action \((X,\beta,\tau)\) is a pre-action \((X',\beta',\tau')\) such
that \(\beta'\) extends \(\beta\) (in particular \( X\subseteq X' \)) and \( \tau' \)
extends \(\tau\). 
Conversely, \((X,\beta,\tau)\) is the \textbf{restriction} to \( X\subseteq X' \) of a pre-action \((X',\beta',\tau')\) if it is itself a pre-action that is extended by \((X',\beta',\tau')\).
As we will recall in Section \ref{Sect: forest sat construction}, every pre-action can be canonically extended to a saturated pre-action that we call the \textbf{maximal forest saturation}.

Finally, by a \textbf{pointed pre-action}, we simply  mean a couple \( (\alpha,x) \) where \(\alpha\) is a pre-action on a set \(X\) and \(x\in X\).
\subsection{Schreier graphs}
 
To each pre-action \( (X,\beta,\tau) \) we associate a \textbf{Schreier graph} 
 whose underlying vertex set is \(X\), endowed with the following labeled edges:
 \begin{itemize}
	\item for every \( x\in X \), 
	we put a positive \(b\) labeled edge from \(x\) to \(\beta(x)\),
	and an opposite negative \( b\inv \)-labeled edge from \(\beta(x)\) to \(x\);
	\item for every \(x\in \dom \tau\), we put a positive \(t\)-labeled edge from \(x\) to \(\tau(x)\), and an opposite \( t\inv \)-labeled edge from \(\tau(x)\) to \(x\).
 \end{itemize}

\subsection{Bass-Serre graphs and \texorpdfstring{\((m,n)\)}{Bass-Serre graphs and (m,n)}-graphs}

To every pre-action \( (X,\beta,\tau) \) we associate another graph coarser 
than the Schreier graph, that we call its Bass-Serre graph, roughly obtained by shrinking \(b\)-orbits
to vertices while keeping track of their cardinality. 
To be more precise, the \textbf{Bass-Serre graph} associated to a pre-action \(\alpha=(X,\beta,\tau)\) of \( \BSo(m,n) \)
is the oriented labeled graph \( \BSe(\alpha) \) defined by
\[
V(\BSe(\alpha))\coloneqq X/\la\beta\ra, \quad
\left\{\begin{array}{l}
	E^+(\BSe(\alpha))\coloneqq\dom(\tau)/\la\beta^n\ra,  \\
	E^-(\BSe(\alpha))\coloneqq\rng(\tau)/\la\beta^m\ra,
\end{array}\right.
\]
where for every \(x\in\dom \tau\) we set
\[\source(x\la \beta^n\ra)\coloneqq x\la \beta\ra,\  \target(x\la \beta^n\ra)\coloneqq x\tau\la \beta\ra,\text{ and }\overline{x\la \beta^n\ra}\coloneqq x\tau \la \beta^m\ra=x\la \beta^n\ra \tau.\]
It is endowed with the \textbf{label map} \( L\colon V(\BSe(\alpha))\sqcup E(\BSe(\alpha))\to \Z_{\geq 1}\cup \{\infty\} \) given by
\[
L(x\la\beta\ra)\coloneqq\abs{x\la\beta\ra},
\quad L(x\la\beta^n\ra)\coloneqq \abs{x\la\beta^n\ra},
\quad L(y\la\beta^m\ra)\coloneqq \abs{y\la\beta^m\ra}.
\]

The labeled oriented graphs obtained as Bass-Serre graphs can be axiomatized
and are called \((m,n)\)-\textbf{graphs} (see \cite[Definition~3.12]{CGLMS-22}): 
every positive (\(t\)-labeled) edge \( e \)  has to satisfy the \textbf{transfer equation},
obtained by writing down the combinatorial condition that the sizes of \( \beta \)-orbits have to satisfy in order for \( \tau \) to be able to take one \(\beta^n\)-suborbit to a \(\beta^m\)-suborbit:
\begin{equation}
	\label{eq:transfert}
	\frac{L(\source(e))}{\gcd(L(\source(e)),n)}=L(e)=\frac{L(\target(e))}{\gcd(L(\target(e)),m)};
\end{equation}
Writing down what this means for \(p\)-adic valuations, we can reformulate this as: for every prime \(p\),
\begin{equation}\label{eq:p adic transfert}
            \max(\abs{L(\source(e))}_p - \abs{n}_p,0)
            = \abs{L(e)}_p
            = \max(\abs{L(\target(e))}_p - \abs{m}_p, 0).
\end{equation}

Note that if one of the two \(\beta\)-orbits is infinite, the other one
has to be as well, and we convene that the transfer equation means that both 
labels are infinite.

Moreover, one can write down the number of \(\beta^m\) and \(\beta^n\) orbits that a 
single \( \beta \)-orbit splits into, yielding the following constraints on the outgoing
and ingoing degree of every vertex \( v \): 
\begin{equation}\label{eq: deg bound by gcd L(v), n m}
	\degout(v)\leq \gcd(L(v),n)\ \ \text{ and } \ \ \degin(v)\leq \gcd(L(v),m), 
\end{equation}
where \(\gcd(k,\infty)=k\) for every nonzero \( k\in\Z \).
Note that a pre-action is saturated iff all the vertices \(v\) of its Bass-Serre graphs
are saturated, i.e. they satisfy
\begin{equation}\label{eq: deg max for saturated}
	\degout(v)= \gcd(L(v),n)\ \ \text{ and } \ \ \degin(v)= \gcd(L(v),m).
\end{equation}
In our previous paper, we showed the expected fact that \((m,n)\)-graphs do 
come from \(\BSo(m,n)\)-actions, but more interestingly that 
any \((m,n)\)-graph extension of the Bass-Serre graph of a pre-action
comes from some extension of the pre-action (see~\cite[Thm.~4.13]{CGLMS-22}). 
However, the results of the present paper won't make use of these facts, 
as we rather directly work at the level of pre-actions.

\subsection{Morphisms}

A \defin{morphism of pre-actions} from \( \alpha_1 = (X_1,\beta_1, \tau_1) \)
to \( \alpha_2 = (X_2,\beta_2, \tau_2) \)
is a map \( \varphi:X_1 \to X_2 \) such that one has 
\( \varphi(x\beta_1) = \varphi(x)\beta_2 \) for all  \( x\in X_1 \)
and \(\varphi(x\tau_1) = \varphi(x)\tau_2 \) for all \( x\in \dom(\tau_1) \).
It is an \defin{isomorphism} if it is bijective and \( \varphi\inv : X_2 \to X_1 \) is also a morphism.

Notice that a morphism \( \varphi:(X_1,\beta_1,\tau_1)\to(X_2,\beta_2,\tau_2) \) 
always maps \(\dom(\tau_1)\) into \(\dom(\tau_2)\).
It is an isomorphism if and only if it is bijective
and maps \( \dom(\tau_1) \) \textit{onto} \( \dom(\tau_2) \).
Also note that every pre-action (iso)morphism induces a labeled graph (iso)morphism 
at the level of Schreier graphs.
The same is true at the level of Bass-Serre graphs, except that morphism
may fail to respect the label map. They are nevertheless graph morphisms. 

\subsection{Connection to Bass-Serre theory}\label{Subsect: connection to Bass-Serre theory}

Finally, we recall the connection to Bass-Serre theory: 
the standard splitting of \(\Gamma = \BSo(m,n)\) as an HNN extension
gives rise to a left $\Gamma$-action on a tree $\Tree$ \cite[Chapter I, §~5]{serre_Trees_1980},
called the \defin{Bass-Serre tree}.
By definitons, this tree coincides with the Bass-Serre graph \(\BSe(\Gamma\curvearrowleft\Gamma)\)
associated to the free transitive
\( \Gamma\)-action by right-translation on itself.




The fact that \(\mathcal T\) is a tree is tightly connected to the uniqueness of
normal forms in HNN extensions, which depend on the choice of coset representatives. 
Here, a natural choice is to take \(\{b^k\colon 0\leq k\leq m-1\}\)
and \(\{b^k\colon 0\leq k\leq n-1\}\)
as coset representatives, and then by \cite[Chapter IV, Theorem 2.1]{lyndon_combinatorial_2001}, every \(\gamma\in\BSo(m,n)\) can be uniquely written in \textbf{normal form}
\[
\gamma=b^{k_1}t^{\varepsilon_1}\cdots b^{k_d}t^{\varepsilon_d}b^k
\]
where \(d\geq 0\), \(k\in\Z\), \(\varepsilon_i=\pm 1\) for \(1\leq i\leq d\), and: 
\begin{itemize}
    \item $\varepsilon_i=+1$ implies $0\leq k_i\leq n-1$,
    \item $\varepsilon_i=-1$ implies $0\leq k_i\leq m-1$,
    \item there is no subword of the form $t^\varepsilon b^0t^{-\varepsilon}$.
\end{itemize}
Note that the case 
$d=0$ corresponds to elements in $\la b\ra$.

Going back to the Bass-Serre tree \(\Tree\),
given a transitive right
\(\BSo(m,n)\)-action \(\Lambda \bs \Gamma \curvearrowleft\Gamma \), one can form the quotient graph 
\(\Lambda\bs\Tree\) whose vertices are thus double cosets 
\(\Lambda\bs\Gamma /\la b\ra \) and thus identify naturally to the vertex set
of the Bass-Serre graph of \(\Lambda\bs\Gamma \curvearrowleft\Gamma\). It is not hard to check that a similar identification holds at the edge level: the Bass-Serre graphs 
that we are considering can simply be stated as quotients of the Bass-Serre tree, suitably 
decorated by the label map\footnote{These labels also arise naturally from the description of \(\Lambda\) as the fundamental group of a graph of group, see \cite[Proposition~2.9 and Proposition~3.34]{CGLMS-22} for details.}. However, they also make sense for pre-actions, making them more flexible.

Finally, given a right action \(\alpha\) on a set \(X\) and \(x\in X\), we have a unique \(\Gamma\)-equivariant map \(\pi_x:\Gamma\to X\) such that \(\pi(\id)=x\), 
which when viewed as a pre-action morphism yields a graph morphism from 
\(\BSe(\Gamma\curvearrowleft\Gamma)=\Tree\) to \(\BSe(\alpha)\).
As we will see, many of our arguments will be stated at the level 
of this graph morphism, which takes a particularly nice form on the forest part of \(\alpha\) when \(\alpha\) is 
a maximal forest saturation (as defined in Section~\ref{Sect: forest sat construction}).

\subsection{Phenotype}\label{sec:phenotype}

Given a \( \BSo(m,n) \)-action (or more generally pre-action), the transfer equation 
that the edges of its Bass-Serre graph has to satisfy
induces a natural equivalence relation on integers, defined as the smallest equivalence relation \(\sim\) such that for all \(L,L'\in\Z_ {\geq 1}\), 
\begin{equation}\label{eq:transfer eq rel}
\frac L{\gcd(L,n)}\sim\frac{L'}{\gcd(L',m)}.
\end{equation}
A basic result from our first paper is the description of this equivalence relation 
as \emph{having equal phenotype} \cite[Proposition~4.6]{CGLMS-22}. The phenotype of an integer \( L\in\Z_{\geq 1} \) is defined as follows: 
\begin{itemize}
	\item We first let \(\primes_{m,n}(L)\coloneqq \left\lbrace p\in\primes\colon\abs{m}_p=\abs{n}_p\text{ and }\abs{L}_p>\abs{n}_p\right\rbrace \).
	\item Then the \( (m,n) \)-\defin{phenotype} of \( L \) is the following positive integer:
	\[\Phe_{m,n}(L)\coloneqq \prod_{p\in \primes_{m,n}(L)} p^{\abs{L}_p}.\]
	If \(L=\infty\), we set \(\Phe_{m,n}(L)\coloneqq\infty \).
\end{itemize}

The phenotype $\PHE_{m,n}(\Lambda)$ of a subgroup $\Lambda\leq \BSo(m,n)$ is obtained as the phenotype 
$\Phe_{m,n}([\la b\ra : \la b\ra \cap \Lambda])$ of the index of $\la b\ra \cap \Lambda$ in $\la b\ra$. 

As we will see, the constructions that we present here towards proving
high topological transitivity and high transitivity still rely on constructions similar 
to \cite[Theorem~4.17]{CGLMS-22} which allow us to connect vertices of equal phenotype, but we carry out such constructions directly
at the level of pre-actions  since we need to control exactly where points are taken.
These constructions will be carried out in details so that no knowledge of the proofs from \cite{CGLMS-22} is required.

Finally, recall from \cite[Theorem~A]{CGLMS-22} that 
the perfect kernel of \(\Sub(\BSo(m,n))\) is
the space 
\[
\PK(\Sub(\BSo(m,n))=
\{\Lambda\leq \BSo(m,n)\colon \Lambda\bs \Tree\text{ is infinite}\},
\]
where $\Tree$ is the standard Bass-Serre tree of $\BSo(m,n)$,
and that the phenotype map yields a partition 
\(\PK(\BSo(m,n))= \bigsqcup_{q\in\mathcal Q_{m,n}} \PK_q\) each of whose pieces 
\(\PK_q\coloneqq \PK(\BSo(m,n))\cap \PHE_{m,n}\inv(q)\) is  topologically transitive \cite[Theorem~B]{CGLMS-22}. 
If we denote the remaining piece in the space of infinite index subgroups by 
\(\mathcal{C}_{\infty}\coloneqq\Sub_{[\infty]}(\BSo(m,n))\smallsetminus \PK(\BSo(m,n))\),
then \(\mathcal{C}_{\infty}\) is empty when \(\abs{m}\neq\abs{n}\), and otherwise
is it a countable open subset of \(\PHE\inv(\infty)\).

\subsection{High faithfulness and proof of Proposition 
\ref{prop-intro: faithfulness according to phenotype}}

\begin{definition}
	Let \(\Gamma\) be a countable group.
    An action $X\curvearrowleft \Gamma$ is \textbf{highly faithful} if for any 
    finite set $F\subseteq\Gamma\smallsetminus\{1\}$, there
    is \(x\in X\) such that 
    for all \(\gamma\in F\), we have \(x\cdot \gamma\neq x\).
\end{definition}

Clearly every highly faithful action is faithful. 
Moreover, \(\Lambda\bs\Gamma\curvearrowleft\Gamma\) is highly faithful
iff there is a sequence of conjugates of \(\Lambda\) converging to 
the trivial subgroup \(\{\id\}\) (equivalently, \(\Lambda\) is 
not a confined subgroup, see \cite{leboudecConfinedSubgroupsHigh2022}).
We can now state and prove a slight strengthening 
of Proposition~\ref{prop-intro: faithfulness according to phenotype}.

\begin{proposition}\label{prop: faithfulness according to phenotype}
Let $\abs{m},\abs{n}\geq 2$, let \(\Gamma=\BSo(m,n)\).
The set of highly faithful actions forms a dense \( G_\delta \) set inside $\PK_q$ 
for every phenotype $q$ when $\abs{m} \neq \abs{n}$ 
and for $q=\infty$ when $\abs{m} = \abs{n}$; otherwise, $\PK_q$ contains no faithful action. 
Moreover, if $\Lambda$ belongs to $\mathcal{C}_\infty$, 
then the action $\Lambda\bs \BSo(m,n)\curvearrowleft \BSo(m,n)$ is not faithful.
\end{proposition}

\begin{proof} [Proof of Proposition~\ref{prop: faithfulness according to phenotype}]
Observe that a subgroup $\Lambda$ gives rise to a highly faithful transitive action if and only if its orbit (under the action by conjugation) accumulates on $\{\id\}$.
 The set of subgroups corresponding to highly faithful actions is thus the \(G_\delta\) subset obtained from the basic open neighborhoods $\mathcal{V}(\emptyset, F)$ of $\{\id\}$ as follows:
 \[\bigcap_{F\Subset \Gamma\smallsetminus \{\id\}}\bigcup_{\gamma\in \Gamma}\mathcal{V}(\emptyset, F)\cdot \gamma.\]

When the phenotype is infinite, by topological transitivity of the action on $\PK_\infty$ (\cite[Theorem 5.14]{CGLMS-22}), there is a dense \( G_\delta\) subset of $\PK_\infty$, consisting of  dense orbits, thus of orbits accumulating on $\{\id\}\in \PK_\infty$.
For $\abs{m}\neq \abs{n}$, the trivial subgroup  $\{\id\}$ belongs to the closure of each $\PK_q$ (for any phenotype) (by \cite[Theorem D (2)]{CGLMS-22}). 
Thus every subgroup with dense orbit in $\PK_q$ 
has its orbit accumulating on $\{\id\}$.
By topological transitivity (\cite[Theorem 5.14]{CGLMS-22}), 
these dense orbits form a dense \(G_\delta\) subset of $\PK_q$. 
On the contrary, when $\abs{m}= \abs{n}$, every subgroup of finite phenotype $q$ contains the non-trivial normal subgroup $\lala b^{s(q,m,n)}\rara$
\cite[Theorem~5.20 (3)]{CGLMS-22}. 
The corresponding transitive actions are not faithful.

It remains to consider the case $\Lambda\in\mathcal{C}_\infty$, i.e.\ the case of infinite index subgroups that do not belong to the perfect kernel (which occurs only when \(\abs{m}=\abs{n}\)). 
This means that $\Lambda\bs \Gamma\curvearrowleft \Gamma$ has finitely many $b$-orbits,
each of which being infinite since $\Lambda$ has infinite index.
In particular, there are finitely many \(b^m\)-orbits. Let $k\in\N$
be the number of such orbits. 
Consider the non-oriented Schreier graph of the \(b^m\)-action: it consists of \(k\) bi-infinite lines, and observe that 
since \(tb^mt\inv=b^{\pm m}\) and \(b\) commutes with \(b^m\), $\Gamma$
acts by automorphism on this graph.
Since the automorphism group of this non-oriented Schreier graph
is the amenable group
\((\Z\rtimes \Z/2\Z)^k\rtimes \Sym(k)\)
and \(\Gamma\) is not amenable,
the kernel of the action $\Lambda\bs \Gamma\curvearrowleft \Gamma$ is not trivial as wanted.
\end{proof}

\section{Pre-actions and forest saturations} \label{Sect: forest saturations}

\subsection{Forest saturations of graphs and of pre-actions}\label{Sect: forest sat construction}
	In this section, we carry out some constructions from \cite{CGLMS-22}
	in a wider context. 
	The goal is to extend any pre-action to an action, or equivalently to a saturated pre-action. 

	In order to do this, it is natural to extend the Bass-Serre graph
	by connecting one vertex which is not yet saturated, i.e.~which does not satisfy 
	Equation \eqref{eq: deg max for saturated}, to a new
	vertex.
	This new vertex must have a label which respects the transfer equation 
	\eqref{eq:transfer eq rel}, but this leaves out several
	possibilities for the label in general, which is why we introduce the following notion.

	\begin{definition}
	    A \textbf{transfer rule} is a map 
	    \[
	        \trrule: (\Z_{\geq 1}\cup\{\infty\})\times \{+,-\}\longrightarrow \Z_{\geq 1}\cup\{\infty\}
        \]
	    such that for all \(L\in \Z_{\geq 1}\cup\{\infty\}\) and \(\varepsilon\in\{\pm\}\), 
		if we set \(L'=\trrule(L,\varepsilon)\)
	    then the transfer equation is satisfied, namely
	    \[
			\frac{L}{\gcd(L,n)} = \frac{L'}{\gcd(L',m)} \text{ if } \varepsilon = +
			\quad ; \quad
			\frac{L}{\gcd(L,m)} = \frac{L'}{\gcd(L',n)} \text{ if } \varepsilon = - .
		\]
	\end{definition}
	
	Important examples of transfer rules are the \defin{maximum rule} and the \defin{minimum rule}, 
	where \( \trrule(L,\varepsilon) \) is the maximum \( L' \), respectively minimum \( L' \), 
	such that the above relation is satisfied.
	The maximum rule can explicitly be described by
	\begin{equation}\label{eq: max rule}
\trrule(L,\varepsilon) = \abs{m}\frac{L}{\gcd(L,n)} \text{ if } \varepsilon = +
			\; ; \;
			\trrule(L,\varepsilon) = \abs{n}\frac{L}{\gcd(L,m)} \text{ if } \varepsilon = - .
	\end{equation}

    Let us now recall the notion of forest saturation introduced in \cite[Definition~4.21]{CGLMS-22}. 
    
    \begin{definition}\label{def: forest saturation of a graph}
		Let \( \Gc_0 \) be a \( (m,n) \)-graph. A \defin{forest extension} of \( \Gc_0 \) 
		is a \( (m,n) \)-graph $\Gc$ containing $\Gc_0$ and such that:
		\begin{itemize}
			\item the subgraph of $\Gc$ induced by the vertices of $\Gc_0$ is exactly $\Gc_0$;
			\item the subgraph of $\Gc$ induced by the other vertices is a forest $\mathcal{F}$;
			\item each connected component $\mathcal A$ of $\mathcal{F}$ is connected to $\Gc_0$ by a single pair of opposite edges in $\Gc$. The endpoint of these edges which is in $\mathcal A$ is called the \defin{root} of $\mathcal A$. 
		\end{itemize}
		If $\Gc$ is moreover saturated, we call it a \defin{forest saturation} of $\Gc_0$.
	\end{definition}
	
	In the context of Definition \ref{def: forest saturation of a graph}, 
	the \defin{parent} of a vertex $v$ of $\Fc$ 
	is the unique neighbour of $v$ which is closer to $\Gc_0$ than $v$.
	This parent is in $\Gc_0$ if $v$ is a root;
	otherwise, it is the neighbour of $v$ closer to the root of the connected component of $\Fc$.
	We say also that $v$ is a \defin{child} of its parent.

	We can finally formalize forest extensions where 
	the label of any vertex \(v'\) in the forest only depends 
	on the label of its parent \(v\) and on the orientation of the edge from the parent to \(v'\).

	\begin{definition}
		Let $\Gc$ be a forest extension of $\Gc_0$ and let $\Fc$ denote the forest.
		We say that the extension \defin{satisfies the transfer rule} \( \trrule\) if
		for any vertex \(v'\) in \( \Fc \), if we denote by \(v\) its parent then 
		\(L(v') = \trrule(L(v),\varepsilon)\),
		where \(\varepsilon\) is the orientation of the edge from \(v\) to \(v'\).
		
		It is called \defin{maximal}, respectively \defin{minimal},
		if it satisfies the maximum rule, respectively the minimum rule.
	\end{definition}
    
    \begin{remark}\label{rmk: degree in max forest sat}
    It follows from Equation \eqref{eq: deg max for saturated} and Equation 
    \eqref{eq: max rule} that in any maximal forest saturation, if \(v'\)
    is a child of \(v\) and \(e\) is the unique edge with source \(v\) and target \(v'\),
    then:
    \begin{itemize}
        \item if \(e\) is positive, then the incoming degree of \(v'\) is \(\abs{m}\);
        \item if \(e\) is negative, then the outgoing degree of \(v'\) is \(\abs{n}\).
    \end{itemize}
    \end{remark}
    
	These definitions extend naturally to pre-actions by considering their Bass-Serre graphs,
	for instance a \textbf{maximal forest saturation of a pre-action} \(\alpha_0\) is
	a pre-action \(\alpha\) extending \(\alpha_0\) whose Bass-Serre graph is 
	a maximal forest saturation of the Bass-Serre graph of \(\alpha\). 
	We can now state the main result of this section.

	\begin{theorem}\label{thm: forest saturation ruled}
		Let \( \alpha_0 \) be a pre-action and
		let $\trrule$ be a transfer rule. Then 
		\( \alpha_0 \) admits a forest saturation \( \alpha \) satisfying $\trrule$,
		and any two such saturations are isomorphic.
		In particular, up to isomorphism, \( \alpha_0 \) admits a unique maximal forest saturation and a unique minimal forest saturation.
	\end{theorem}

	The proof of the existence relies on Zorn's lemma, the key point being to constructively
	extend non-saturated pre-actions by adding one orbit. We isolate this as the following construction.\\

	Let \( \alpha_1 = (Y, \beta_1, \tau_1) \) be a pre-action that is not saturated,
	let \(Y'\) be an infinite set disjoint from \(Y\).
	Since \(\alpha_1\) is not saturated, one of the following cases must occur.
\begin{itemize}
\item \(Y \smallsetminus \dom(\tau_1)\neq \emptyset\).
		Let us fix \(x\in Y \smallsetminus \dom(\tau_1)\).
		Let \(L\) be the cardinality of the \(\beta_1\)-orbit of \(x\).
		Let \(O'\subseteq Y'\) have cardinality \(L'=\trrule (L,+)\), 
		let  \(\beta\in \Sym(O')\) be an \(L'\)-cycle
		and define \(\beta_2=\beta_1\sqcup \beta\).
		
		Finally, fix some \(y\in O'\) and define \(\tau: x\la {\beta}^n\ra \to y\la {\beta}^m\ra\)
		as the unique \((\beta^n,\beta^m)\)-equivariant map taking \(x\) to \(y\), namely
		\[
		x\beta^{jn}\tau \coloneqq y {(\beta')}^{jm} \text{ for all } j\in \Z.
		\]
		This makes sense precisely because the transfer rule \( \trrule \)
		satisfies the transfer equation.
		Letting \(\tau_2=\tau_1\sqcup \tau\), we have that 
		\( \alpha_2=(Y\sqcup O', \beta_2,\tau_2) \) is as desired. 
		We then say that \(\alpha_2\) is a \textbf{positive one point free \(\trrule\)}-extension
		of \((\alpha_1,x)\)
\item  $Y \smallsetminus \rng(\tau_1)\neq \emptyset$.
		We fix this time \(x\in Y \smallsetminus \rng(\tau_1)\) and 
		proceed similarly as before: 
		let \(L\) be the cardinality of the \(\beta_1\)-orbit of \(x\).
		Let \(O'\subseteq Y'\) have cardinality \(L'=\trrule (L,-)\), 
		let  \(\beta\in \Sym(O')\) be an \(L'\)-cycle
		and define \(\beta_2=\beta_1\sqcup \beta\).
		
		Fix some \(y\in O'\) and define \(\tau: x\la \beta^m\ra \to y\la \beta^n\ra\)
		as the unique \((\beta^m,\beta^n)\)-equivariant map taking \(x\) to \(y\), which
		makes sense again thanks to the fact that \( \trrule \)
		satisfies the transfer equation.
		Letting \(\tau_2=\tau_1\sqcup \tau\), we have again that 
		\( \alpha_2=(Y\sqcup O', \beta_2,\tau_2) \) is as desired.
	We then say that \(\alpha_2\) is a \textbf{negative one point free \(\trrule\)}-extension
		of \((\alpha_1,x)\).
\end{itemize}

	\begin{lemma}
		\label{lem: isomorphisms and one orbit extensions}
		Let $\varphi: X_1 \to X_2$ be an isomorphism between 
		non-saturated pre-actions $\alpha_1 = (X_1,\beta_1, \tau_1)$ to $\alpha_2 = (X_2,\beta_2, \tau_2)$.
		Assume $\alpha'_1$ is a one orbit free \( \trrule \)-extension of $(\alpha_1,x_1)$ and
		$\alpha'_2$ is a one orbit free \( \trrule \)-extension of $(\alpha_2,\varphi(x_1))$
		both positive or both negative.
		Then, $\varphi$ extends to an isomorphism from $\alpha'_1$ to $\alpha'_2$.
	\end{lemma}
	\begin{proof}
		Set $\alpha'_i = (X'_i, \beta'_i, \tau'_i)$ for $i=1,2$ and $x_2 = \varphi(x_1)$.
		We treat the positive case. The negative case is similar and left to the reader.

		Recall that $X'_i = X_i \sqcup y_i \la \beta_i' \ra$ for $i=1,2$.
		The orbits of $y_1, y_2$ both have cardinal $M=\trrule(L,+)$ where
		\(L=\abs{x\la\beta_1\ra}=\abs{\varphi(x)\la \beta_2\ra}\).
		Thus, one can extend $\varphi$ to a bijection $\varphi': X'_1 \to X'_2$ by setting
		\[\varphi'(y_1 (\beta'_1)^k)  \coloneqq y_2 (\beta'_2)^k
		\text{ for all } k\in \Z .\]
		It remains to check that $\varphi'$ is an isomorphism
		between $\alpha'_1$ and $\alpha'_2$.
		The relation $\varphi'(x\beta'_1) = \varphi'(x)\beta'_2$, for every $x\in X'_1$, is obvious.
		Moreover, one has
		\[
			\dom(\tau_i') = \dom(\tau_i) \sqcup x_i \la \beta_i^n \ra
			\quad \text{ for } i= 1,2,
		\]
		so that $\varphi'$ maps $\dom(\tau_1')$ onto $\dom(\tau_2')$.
		Finally, for every $k\in \Z$, one has
		\begin{align*}
			\varphi'(x_1 \beta_1^{kn} \tau_1')
			& = \varphi'(x_1 \tau_1' (\beta_1')^{km})
			= \varphi'(y_1 (\beta_1')^{km})
			= y_2  (\beta_2')^{km} \\
			& = y_1 \tau_2' (\beta_2')^{km}
			= \varphi'(x_1)\beta_2^{kn} \tau_2'
			= \varphi'(x_1 \beta_1^{kn}) \tau_2' .
		\end{align*}
		Thus $\varphi'(x\tau_1') = \varphi'(x) \tau_2'$ for all $x\in \dom(\tau_1')$; the proof is complete.
	\end{proof}

	We now have all the tools in order to prove Theorem \ref{thm: forest saturation ruled}.

	\begin{proof}[Proof of Theorem \ref{thm: forest saturation ruled}]
		Let \(\alpha_0\) be a non-saturated pre-action on a set \(X\), we first need to show that 
		\(\alpha_0\) admits a saturated forest extension satisfying the transfer rule \(\trrule \). 
		
		First observe that any forest extension of \(\alpha_0\) must have an underlying set
		which is countable or finite if \(X\) was finite, and of the same cardinality as \(X\)
		if \(X\) was infinite. 
		Let us then fix a set \(Z\) containing \(X\) whose cardinality is strictly greater
		than that of \(X\). 
		
		Zorn's lemma provides us a forest extension \(\alpha=(Y,\beta,\tau)\) of \(\alpha_0\) satisfying \(\trrule\) 
		whose underlying 
		set is a subset of \(Z\) and which is maximal among such extensions. 
		Assume by contradiction \(\alpha\) is not saturated, then by our previous observation 
		\(Z\smallsetminus Y\) is infinite and we can thus construct a positive or negative 
		one point free \( \trrule \)-extension of \(\alpha\) with underlying set contained in \(Z\)
		as explained right after the statement of Theorem \ref{thm: forest saturation ruled}.
		This contradicts the maximality of \(\alpha\), which is thus saturated as desired.

		We now prove uniqueness: assume 
		\( \alpha_1, \alpha_2 \) are forest saturations of \( \alpha_0 \) satisfying \( \trrule \).
		For \( i=0,1,2 \), write \( \alpha_i = (X_i, \beta_i, \tau_i) \).
		Consider the set of partial isomorphisms,
		from a forest extension \( \alpha_1' = (X_1',\beta_1',\tau_1') \) contained in \( \alpha_1 \) 
		to a forest extension \( \alpha_2' = (X_2',\beta_2',\tau_2') \) contained in \( \alpha_2 \),
		that fix \( X_0 \) pointwise.
		This is a non-empty inductive poset.
		Hence, it contains a maximal element \( \varphi:X_1' \to X_2' \) by Zorn's Lemma.
		
		Let us prove by contradiction that $\alpha_1' = \alpha_1$ and $\alpha_2' = \alpha_2$; this will complete the proof.
		If this is  not the case, we may assume that $\alpha_1' \neq \alpha_1$ 
		(the case $\alpha_2' \neq \alpha_2$ is similar).
		In this case, $\alpha_1'$ is non-saturated,
		hence one can find \(x\in X'_1\) and a restriction \(\alpha''_1\) of \( \alpha_1 \)
		which is a one-orbit free \( \trrule \)-extension of \( (\alpha_1',x_1) \).
		As $\varphi: X_1' \to X_2'$ is an isomorphism
		the orbit $\varphi(x) \la \beta_2 \ra$ lies in $X_2'$
		while $\varphi(x) \tau_2^\varepsilon \la \beta_2 \ra$ does not.
		So, the restrictions of $\beta_2, \tau_2$ to 	
		\(X_2'' \coloneqq X_2' \sqcup \varphi(x) \tau_2^\varepsilon \la \beta_2 \ra\)
		give a one-orbit free \(\trrule\)-extension $\alpha_2''$ of $\alpha_2'$, contained in $\alpha_2$.
		Notice that $x \la \beta_1 \ra$ and $\varphi(x) \la \beta_2 \ra$
		have the same cardinal since $\varphi$ is an isomorphism.
		Lemma \ref{lem: isomorphisms and one orbit extensions} now implies
		that $\varphi$ extends to an isomorphism from $\alpha_1''$ to $\alpha_2''$, 
		contradicting the maximality of $\varphi$.
	\end{proof}

	\begin{remark}
		\label{rem: saturation by sequence of one orbit extensions}
		When $\BSe(\alpha_0)$ is countable, one can in fact obtain any forest saturation \( \alpha \)
	 	by a sequence of one orbit free extensions.
		More precisely, there exists a sequence
		\[
			\alpha_0 \subseteq \alpha_1 \subseteq \cdots \subseteq \alpha_n \subseteq \alpha_{n+1} \subseteq \cdots
		\]
		of one orbit free extensions such that $\alpha = \bigcup_{i\geq 0} \alpha_i$.
		To see this, it suffices to enumerate the edges of the forest wisely,
		e.g.~by doing a simultaneous breadth-first exploration of each 
		connected component the forest. 
	\end{remark}

\subsection{Stabilizer of a pointed pre-action}\label{subsect: stab of pre-action}

	Let $\Gamma \coloneqq \BSo(m,n)$.
	Let $\alpha=(X, \beta, \tau)$ be a pre-action of \( \Gamma \).
	Every path \( c \) in $\Sch(\alpha)$ between two vertices is labelled by a word on letters $t, t\inv, b, b\inv$;
	this word represents an element of \( \Gamma \) that we denote $\Psi(c)$.
	Notice that we have $\target(c) = \source(c)\cdot \alpha(\Psi(c))$ in \( X \).

	In particular, for every choice of a basepoint $x_0\in X$, the map $\Psi$ defines a group homomorphism $\Psi_{x_0}:\pi_1(\Sch(\alpha), x_0)\to \Gamma$ from the fundamental group of the connected component of \( x_0 \) to \( \Gamma \).
	\begin{definition}
		\label{stabilizer of pre-action}
		The image of $\Psi_{x_0}$ is called the \( \alpha \)-\defin{stabilizer} of \( x_0 \), and denoted $\Stab_\alpha(x_0)$.
	\end{definition}
	If $x_1$ is another vertex in the same connected component of $\Schreier(\alpha)$, the stabilizers of \( x_0 \) and $x_1$ are conjugate:
	one has
	\[
		\Stab_\alpha(x_0) = \Psi(c) \, \Stab_\alpha(x_1) \, \Psi(c)\inv ,
	\]
	where \( c \) is any path from \( x_0 \) to $x_1$.
	If the pre-action \( \alpha \) is saturated, i.e. it is a genuine \( \Gamma \)-action, 
	then $X\curvearrowleft^\alpha \Gamma$ is isomorphic to $\Stab_\alpha(x_0)\bs \Gamma \curvearrowleft \Gamma$
	and $\Stab_\alpha(x_0)$ is the usual stabilizer of \( x_0 \) with respect to the action \( \alpha \).

	\begin{proposition}[Stabilizer of a one orbit free extension]
	    \label{prop: stabilizer of a one orbit free extension}
		\label{prop: stab any one orbit free extension}
		Let \( \alpha_0 \) be a non-saturated pre-action and
		\( \alpha \) be a one orbit free extension of \( \alpha_0 \) from $(x,L_x)$ to $(y,L_y)$.
		Set $\varepsilon = 1$ if this extension is positive and $\varepsilon = -1$ if it is negative.
		\begin{enumerate}
			\item If \( \alpha_0 \) has finite phenotype, 
			then the \( \alpha \)-stabilizer of the basepoint \( x_0 \) is generated by the \( \alpha_0 \)-stabilizer and the element
			$\Psi(c)\left(t^\varepsilon b^{L_y} t^{-\varepsilon}\right)\Psi(c)^{-1}$,
			where \( c \) is any path from \( x_0 \) to \( x \) in $\Schreier(\alpha_0)$. In formula:
			\[
				\Stab_\alpha(x_0)=\la \Stab_{\alpha_0}(x_0), \Psi(c)\left(t^\varepsilon b^{L_y} t^{-\varepsilon}\right)\Psi(c)^{-1}\ra .
			\]

			\item If the extension \( \alpha \) is maximal, or if the phenotype is infinite,
			then the stabilizers coincide:
			\[
				\Stab_\alpha(x_0)=\Stab_{\alpha_0}(x_0).
			\] 
		\end{enumerate}
	\end{proposition}
	\begin{proof} 
		Up to conjugating the situation by $\Psi(c)$,
		it is equivalent to prove
		\(\Stab_\alpha(x)=\la \Stab_{\alpha_0}(x), t^\varepsilon b^{L_y} t^{-\varepsilon} \ra\) in Item 1,
		respectively \(\Stab_\alpha(x)=\Stab_{\alpha_0}(x)\) in Item 2.
		Clearly, the free group \(\pi_1(\Schreier(\alpha), x)\) is obtained from $\pi_1(\Schreier(\alpha_0), x)$ by adding the following generators:
		\begin{itemize}
			\item if the extension is positive, the cycles \( c_j \), for \(j\in  \Z\), based at \( x \) and labelled by $t b^{jm} t^{-1} b^{-jn}$;
			\item if the extension is negative, the cycles \( c_j' \), for \(j\in  \Z\), based at \( x \) and labelled by $t\inv b^{jn} t b^{-jm}$;
			\item if the phenotype is finite, the cycle \( c \) based at \( x \) and labelled by $t^\varepsilon b^{L_y} t^{-\varepsilon}$.
		\end{itemize}
		(these cases are not mutually exclusive). 
		Notice that all the \(c_j\) have trivial image under $\Psi_{x_0}$,
		hence \(\Stab_\alpha(x)=\Stab_{\alpha_0}(x)\) if the phenotype is infinite, concluding the proof in this case.

		If the phenotype is finite, 
		we get \(\Stab_\alpha(x)=\la \Stab_{\alpha_0}(x), t^\varepsilon b^{L_y} t^{-\varepsilon} \ra\),
		so Item 1 is established.
		Finally, if the extension is moreover maximal, 
		it suffices to show that $t^\varepsilon b^{L_y} t^{-\varepsilon}$ 
		already belongs to $\Stab_{\alpha_0}(x_0)$.
		In the positive case, one has \( L_y=m\frac{L_x}{\gcd(L_x,n)} \), 
		so we have in \( \Gamma \):
		\[\left(t b^{L_y} t^{-1}\right)=(t b^mt^{-1})^{\frac{L_x}{\gcd(L_x,n)} }=(b^n)^{\frac{L_x}{\gcd(L_x,n)} }
		=(b^{L_x})^{\left(\frac{n}{\gcd(L_x,n)}\right)}.\]
		As \(b^{L_x}\) belongs to \( \Stab_{\alpha_0}(x) \), this case is over.
		The negative case is similar.
	\end{proof}

	\begin{corollary}\label{cor: maximal forest sat stabilizer}
		Let $(\alpha_0, x_0)$ be a pointed transitive pre-action
		and \( \alpha \) be the associated maximal forest saturation.
		Then $\Stab_{\alpha}(x_0)=\Stab_{\alpha_0}(x_0)$. In particular, the maximal forest saturation identifies with $\Stab_{\alpha_0}(x_0)\bs \Gamma \curvearrowleft \Gamma$. 
	\end{corollary}
	\begin{proof} 
		As in Remark \ref{rem: saturation by sequence of one orbit extensions},
		the maximal forest saturation is obtained by a sequence
		\[
			\alpha_0 \subseteq \alpha_1 \subseteq \cdots \subseteq \alpha_n \subseteq \alpha_{n+1} \subseteq \cdots
		\]
		of \textbf{maximal} one orbit free extensions such that $\alpha = \bigcup_{i\geq 0} \alpha_i$.
		By Proposition~\ref{prop: stab any one orbit free extension} (item 2), the sequence of natural injections 
		\[\Stab_{\alpha_0}(x_0)\hookrightarrow \Stab_{\alpha_1}(x_0)\hookrightarrow \cdots \hookrightarrow \Stab_{\alpha_n}(x_0)\hookrightarrow \Stab_{\alpha_{n+1}}(x_0)\hookrightarrow \cdots\]
		consists only of identity maps. Thus $\Stab_{\alpha}(x_0)=\Stab_{\alpha_0}(x_0)$.
	\end{proof}

    More generally, Proposition~\ref{prop: stabilizer of a one orbit free extension}
    allows us to describe the stabilizer of any forest saturation $\alpha$ of $\alpha_0$ satisfying a transfer rule:
    just write the saturation as a sequence of one-orbit free extensions as before;
    then Proposition~\ref{prop: stabilizer of a one orbit free extension}
    provides a list of generators to add to \(\Stab_{\alpha_0}(x_0)\) so as to obtain \(\Stab_{\alpha}(x_0)\).

\subsection{Leaving finite pre-actions through their maximal forest saturation}

The following technical lemma underlies both our results 
on high transitivity and on high topological transitivity. 
It allows us to get out of any finite transitive pre-action
through its maximal forest saturation, with some 
additional control on the labels that we end up at.

Even better, we can leave \emph{simultaneously} finitely many
finite transitive pre-actions.
The key observation, already apparent in \cite[Claim~3.5]{moonHighlyTransitiveActions2013},
is that once  a reduced path has arrived in the maximal forest saturation,
it will stay there (see Claim \ref{claim: no backtracking once in Fi}).
This allows us to inductively construct the path we seek, working 
out each pre-action one by one.

\begin{lemma}
\label{lem: leaving the seed uniformly}
Let $(\alpha_1^0,x_1),$$\dots,(\alpha_k^0,x_k)$ 
be transitive non-saturated pointed pre-actions with finite Bass-Serre graphs.
Let $\alpha_1$,\dots, $\alpha_k$ be the maximal forest saturations of these pre-actions, with underlying sets $X_1,\dots, X_k$.
Recall that $\Tree$ is the Bass-Serre tree of $\BSo(m,n)$.
For every $i\in\{1,\dots,k\}$, consider the graph morphism $\pi_i: \Tree\to \BSe(\alpha_i)$ associated to the \( \Gamma \)-equivariant map $\hat\pi_i:\Gamma\to X_i$ taking $\id$ to $x_i$. 

Then there is a reduced path \( c \) in $\Tree$ from the base vertex $\la b\ra$ to a vertex $g \la b \ra$, whose last edge $f$ is positive and such that for all $i\in\{1,\dots,k\}$:
\begin{enumerate}[label=(\arabic*)]
    \item\label{item: separation in mfs} the edge $\pi_i(f)$ separates $\BSe(\alpha_i^0)$ from $\vv_i \coloneqq \pi_i(g\la b\ra)=x_i\alpha_i(g\la b\ra)$;
       \item\label{item: label in mfs} if the phenotype $q_i$ of $\alpha_i^0$ is finite, then the label of the terminal vertex $\vv_i$ satisfies
    \\$\bullet$
    $\abs{L(\vv_i)}_p=\abs m_p$ for all prime $p$
        such that $\abs{m}_p<\abs{n}_p$,
        and
       \\$\bullet$ $\abs{L(\vv_i)}_p=\max(\abs{q_i}_p,\abs m_p)$ for all prime $p$
        such that $\abs{m}_p=\abs{n}_p$.  
\end{enumerate}
Moreover, any extension of \( c \) using only positive edges satisfies the same conclusions.
\end{lemma}

\begin{proof}
    We first take care of condition \ref{item: separation in mfs}.
    For every $i\in\{1,...,k\}$, denote by $\mathcal{F}_i$ the forest that was added to the
    Bass-Serre graph of $\alpha_i^0$
    when building its maximal forest saturation $\alpha_i$. 
    We rely on the following crucial claim.
    
    \begin{claim}\label{claim: no backtracking once in Fi}
                Let $i\in\{1,\dots,k\}$.
                If $\pathinT$ is a reduced path in $\Tree$ starting at the base vertex $\la b \ra$ and whose projection $\pi_i(\pathinT)$ contains an edge $\pi_i(e)$ in $\mathcal{F}_i$, then the subpath of $\pi_i(\pathinT)$
                starting with $\pi_i(e)$ has no backtracking, hence it follows a geodesic in $\mathcal{F}_i$.
    \end{claim}
            \begin{cproof}
                Without loss of generality, we may and will assume that $e$ is the first edge in $\pathinT$ such that $\pi_i(e)$ lies in $\mathcal{F}_i$.
                This implies that the source $\source(\pi_i(e))$ lies in $\BSe(\alpha_{i}^0)$.
                
                Recall from \cite{CGLMS-22} (see the proof of Lemma 4.21),
                that the maximal forest saturation is constructed by iteratively adding at each step as many new vertices $v$ as possible so that 
                \begin{itemize}
                    \item $v$ is connected to the graph constructed at the previous step by a unique edge $a$ and   
                    
                    \item the label of $v$ satisfies $L(v) = m L(a)$ if $a$ is positive and $L(v) = n L(a)$ if $a$ is negative.
                    
                \end{itemize}
                The second item implies $\degin{v} = m$ if $a$ is positive, and  $\degout{v} = n$ if $a$ is negative (however, one may fail to have both $\degin{v} = m$ and  $\degout{v} = n$).
                Since $\Tree$ is a regular tree with incoming degree $m$ and outgoing degree $n$, this may be reformulated as: for any pull-back $\tilde v\in \Tree$ of $v$, the map $\pi_i$ is injective on the set $E_{\tilde v, a}$ of edges attached to $\tilde v$ which have the same orientation as $a$. 
                
                Assume by contradiction that $\pi_i(\pathinT)$ does backtrack at some point.
                Let $e_1, e_2$ be the first pair of consecutive edges of $\pathinT$ after $e$ whose $\pi_i$-images form a backtracking and 
                let $\tilde v \coloneqq \target(e_1) = \source(e_2)$ be their common vertex. 
                Let $v \coloneqq \pi_i(\tilde v)=\target(\pi_i(e_1)) = \source(\pi_i(e_2))$. 
                Since $\pi_i(\pathinT)$ has no backtracking between $\pi_i(e)$ and $\pi_i(e_1)$,
                the source $\source (\pi_i(e_1))$ is in the same connected component as the image of the basepoint in $\BSe(\alpha_i)\smallsetminus  \pi_i(e_1)$, 
                i.e.\ the target $v = \target (\pi_i(e_1))$ appears after $\source (\pi_i(e_1))$ in the construction of the maximal forest saturation
                and $\pi_i(e_1)$ is the unique edge connecting $v$ to the previously constructed stages.
                
                Since $\pi_i(e_1)=\pi_i(\bar e_2)$, the edges $e_1,\bar e_2$ have the same orientation 
                and they belong to $E_{\tilde v, e_1}$, a set on which $\pi_i$ is injective.
                This contradicts the fact that $\pathinT$ is reduced. 
            \end{cproof}

            We can now construct 
            the path $\pathinT$ satisfying condition 1 as the last term of an inductive construction $(\pathinT_i)_{i=0}^{k}$ with $\pathinT_{i}$ extending $\pathinT_{i-1}$
            so that for all
            $i\in \{1, \dots, k\}$ and $j\in \{1, \ldots, i\}$, the path $\pi_j(\pathinT_i)$ ends in $\mathcal F_j$.
            The first term $\pathinT_0$ is chosen to be any edge with origin $\la b \ra$ in $\Tree$.
            For the inductive step, assume $\pathinT_i$ has been constructed for some  $i< k$. 
            Let $e_i$ be the last edge of $\pathinT_i$.
            By \cite[Lemma 2.17]{fima_characterization_HT_2022}
            we can extend $\pi_{i+1}(e_i)$ to a reduced path $\pi_{i+1}(e_i) q_{i+1}$ in $\BSe(\alpha_{i+1})$ 
            whose terminal edge disconnects the graph $\BSe(\alpha_{i+1})$ and points towards a tree.
            Since $\BSe(\alpha_{i+1})$ has no vertex of degree $\leq 1$, this tree is unbounded. 
            As the complement of $\mathcal{F}_{i+1}$ is finite, we can extend $q_{i+1}$ if necessary so that it ends in $\mathcal{F}_{i+1}$.

            Let $e_i\tilde q_{i+1}$ be a lift in $\Tree$ of $\pi_{i+1}(e_i) q_{i+1}$ starting by $e_i$, 
            then $\pathinT_{i+1} \coloneqq \pathinT_i \tilde q_{i+1}$ has the following properties:
            \begin{itemize}
                \item $\pathinT_{i+1}$ is a reduced path;
                \item $\pi_{i+1}(\pathinT_{i+1})$ ends in $\mathcal{F}_{i+1}$ by construction;
                \item for every $j\in \{1, \ldots, i\}$, the path $\pi_j(\pathinT_{i+1})$ ends in $\mathcal{F}_j$, by Claim~\ref{claim: no backtracking once in Fi}.
            \end{itemize}

            This concludes the construction of $\pathinT = \pathinT_{k}$ if the last edge of $\pathinT_{k}$ is positive. 
            If not, we extend $\pathinT_{k}$ to $\pathinT=\pathinT_{k}f$ by adding a last positive edge $f$ without reduction: this is possible since the in-degree and the out-degree of every vertex in $\Tree$ is $\geq 2$ when $\vert m\vert ,\vert n\vert>1$. Claim \ref{claim: no backtracking once in Fi} ensures that $\pathinT_{k}f$ continues to satisfy condition 1, namely that $\pi_i(f)\in \mathcal{F}_i$ for all $i\in\{1,\dots,k\}$.

            If the phenotype is infinite, we are done since condition 2 is empty in this case.
            Let us thus assume that the phenotype is finite. 
            Notice that by Claim \ref{claim: no backtracking once in Fi} any reduced extension of $\pathinT$  keeps satisfying condition 1. Thus it suffices to extend $\pathinT$ to a longer reduced path such that it
            also satisfies condition 2.

        Consider first a geodesic ray $r$ consisting of positive edges and starting at the terminal vertex $\target(\pathinT)$ of $\pathinT$.
        The path $\pathinT r$ and its projection $\pi_i(\pathinT r)$ are reduced,
        since the last edge in $\pathinT$ was positive.
        For every $i$, the projection $\pi_i(r)$ lies in $\mathcal{F}_i$ and for each edge $e$ of $\pi_i(r)$, 
        the source of $e$ was built before its target in the maximal forest saturation.

        For each edge $e$ of $\pi_i(r)$, and for any prime $p$ such that $\abs{m}_p\leq\abs{n}_p$, we claim that:
        \begin{equation}\label{eq: transfermax}
        \abs{L(\target(e))}_p=\max\big(\abs m_p,\abs{L(\source(e))}_p-(\abs n_p-\abs m_p)\big).
        \end{equation}
        Indeed, recall the label transfer equation \eqref{eq:p adic transfert}: 
               \[
            \max(\abs{L(\source(e))}_p - \abs{n}_p,0)
            = \abs{L(e)}_p
            = \max(\abs{L(\target(e))}_p - \abs{m}_p, 0).
        \]
        If $\abs{L(\source(e))}_p \leq \abs{n}_p$, then  $\abs{L(e)}_p = 0$, hence the maximality in the construction of the maximal forest saturation gives
        $\abs{L(\target(e))}_p = \abs{m}_p$.
        If $\abs{L(\source(e))}_p > \abs{n}_p$, then $\abs{L(e)}_p = \abs{L(\source(e))}_p-\abs n_p > 0$, hence 
        Equation \eqref{eq:p adic transfert} gives 
        $\abs{L(\target(e))}_p = \abs{L(\source(e))}_p-\abs n_p+\abs m_p$.
        This proves Equation \eqref{eq: transfermax}.

        Now, for any edge $e$ in $\pi_i(r)$ and all prime $p$
        such     that $\abs{m}_p=\abs{n}_p$,
        we have
 \begin{align*}  \abs{L(\target(e))}_p 
            & = \max\big(\abs m_p,\abs{L(\source(e))}_p\big),
            \text{ by Equation \eqref{eq: transfermax}}\\
            & = \max\big(\abs m_p,\abs{q_i}_p\big), 
            \text{ by definition of the phenotype}. 
        \end{align*}
        Finally, Equation \eqref{eq: transfermax}  yields that for any vertex $g\la b\ra$ far enough in the ray $r$, for all prime $p$
        such that $\abs{m}_p<\abs{n}_p$ and  all $i \in \{1, \ldots, k\}$
        we have
        $\abs{L(\pi_i(g\la b\ra))}_p=\abs m_p$.        
    
        So replacing $\pathinT$ by $\pathinT r'$, where $r'$ is any sufficiently long
        initial segment of $r$, we get that $\pathinT$ is the desired path.
        
        The last statement about extensions of $\pathinT$ by positive edges
        follows from Equation \eqref{eq: transfermax}, using the same argument as before.
        This finishes the proof of Lemma \ref{lem: leaving the seed uniformly}.
\end{proof}

\section{High topological transitivity results}

Recall that for any phenotype $q$ (finite or not), \( \PK_q(\BSo(m,n))  \) is a perfect compact space \cite[Remark 5.12]{CGLMS-22}.

\begin{theorem}[multiple topological transitivity]
	\label{thm: multiple topological transitivity}
	
	Let \( m,n \) be integers such that ${\abs m}, \abs n\geq 2$. Then for every phenotype $q\in \QQ_{m,n}$ the action by conjugation of \( \BSo(m,n)  \) on the invariant subspace \( \PK_q(\BSo(m,n))  \) is highly topologically transitive.
\end{theorem}

\begin{proof}
    	Let \( d \) be a positive integer and let $U_1,\dots, U_{2d}$ be nonempty open subsets of \( \PK_q(\BSo(m,n))  \).
	We want to prove that there exists $\gamma \in \BSo(m,n)$ such that,
	\[
		\forall i\in\{1,\dots,d\}, \quad U_i\cdot \gamma\cap U_{i+d}\neq\emptyset
	\]
	We split this proof into several steps.
	
	\paragraph{Step 1. Preliminary reductions.}	
	Up to shrinking the $U_i$'s, we may as well assume they are all of the form
	\[
		U_i=\Nc([\alpha_i,x_i],R)
	\]
	for some $R>0$, 
	i.e. all the (classes of) pointed transitive actions with the same $R$-ball of Schreier graph as some pointed transitive action $(\alpha_i, x_i)$ on an infinite set $X_i$.
	Recall that $X_i$ contains infinitely many \( b \)-orbits 
	(in other words, $\BSe(\alpha_i)$ is infinite)
	since we are in the perfect kernel.
	
	Then, for any $i\in\{1,\dots, 2d\}$, 
	let $B_i$ be the union of the \( b \)-orbits that meet the ball $B(x_i, R)$ in the Schreier graph of $\alpha_i$ 
	and take the restriction $\alpha_i^0 \coloneqq {\alpha_i}_{\restriction B_i}$.
	Notice that, as a strict subgraph of an $(m,n)$-graph, $\BSe(\alpha_i^0)$ has a non saturated vertex, thus $\alpha_i^0$ is not an action.
	We may moreover assume that 
	$\alpha_i$ is the maximal forest-saturation,
	(as defined in Section \ref{Sect: forest saturations}) of $\alpha_i^0$.
	Indeed, $B(x_i, R)$ is contained in $B_i$. 
	Therefore, replacing $\alpha_i$ by the maximal forest saturation of $\alpha_i^0$
	doesn't change the ball $B(x_i, R)$ in the Schreier graph, 
	hence doesn't change $\Nc([\alpha_i,x_i],R)$.

	\paragraph{Step 2. Going out of balls ``uniformly''.}	
	
	We are now in the situation of Lemma \ref{lem: leaving the seed uniformly}.
	So, let us
	consider the graph morphism $\pi_i:\Tree \to \BSe(\alpha_i)$ and
	pick a reduced path $\pathinT$ in $\Tree$, from $\la b\ra$ to a vertex $g \la b \ra$, given by this lemma.
	Notice that, up to changing the representative $g$ of the coset $g \la b \ra$,
	we may assume that its normal form $w$,
	as defined in Section~\ref{Subsect: connection to Bass-Serre theory}
	ends on the right by $t^{\pm 1}$.

	Notice also that $w$ defines a reduced path from $\id$ to $g$ in $\Cayley(\BSo(m,n))$.
	This path projects onto $\pathinT$, since $\pathinT$ is the unique reduced path from $\la b\ra$ to $g \la b \ra$ in $\Tree$ and $w$ is a normal form.
	Observe that, since the last edge of $\pathinT$ is positive,
	$w$ ends by a positive power of \( t \) .
	
	In $\Schreier(\alpha_i)$, consider the path from $x_i$ labeled by $w$, and its target $y_i\coloneqq x_i \alpha_i(g)$.
	Notice that its projection in $\BSe(\alpha_i)$ is the path $\pi_i(c)$.
	Set moreover $v_i \coloneqq \pi_i(\target(c))$, which is also the projection of $y_i$.
	
	\paragraph{Step 3. From actions to pre-actions.}
	    
	Let $\Gc_i$ be the subgraph of $\BSe(\alpha_i)$ consisting of the union of $\BSe(\alpha_i^0)$ with $\pi_i(c)$.
	By Condition \ref{item: separation in mfs} from Lemma \ref{lem: leaving the seed uniformly}, the vertex
	$v_i$ is disconnected from the rest of $\Gc_i$ by the last edge of $\pi_i(c)$.
	Moreover, by the maximal transfer rule, as the last edge of \( c \) is positive, 
	the label $L(v_i)$ is a multiple of $\abs{m}$ and hence $v_i$ has ingoing degree $\abs{m}$.
	
	By contrast, we have seen that $v_i$ has outgoing degree $1$ in $\Gc_i$.

    Consider the subgraph of $\Sch(\alpha_i)$ consisting of the \( b \)-orbits that are shrunk to vertices of $\Gc_i$ 
    and all the edges that are shrunk to edges of $\Gc_i$. 
    It is the Schreier graph of a pre-action that we denote $\xi_i$. 
    Observe that $\xi_i$ extends $\alpha_i^0$ and that its Bass-Serre graph is $\Gc_i$.

    The former discussion about $\Gc_i$ shows that
    the $\xi_i(b)$-orbit of $y_i$ does not intersect the domain of $\xi_i(t)$ 
    and intersects the target of $\xi_i(t)$ only along the $\xi_i(b)^{m}$-orbit of $y_i$.
	
	\paragraph{Step 4. Conclusion in case of infinite phenotype.}
	
	In this case, the orbits $y_i \xi_i(\la b \ra)$, for $i\in \{1, \ldots, 2d\}$, are moreover infinite.
	For each $i\in \{1, \ldots, d\}$,
	we shall weld $\xi_i$ and $\xi_{i+d}$ into a new transitive preaction $\eta_i$ as follows.
	The domain of $\eta_i$ is the disjoint union of those of $\xi_i$ and of $\xi_{i+d}$ and
	the bijection $\eta_i(b)$ is the disjoint union of $\xi_i(b)$ and $\xi_{i+d}(b)$.
	In order to define $\eta_i(t)$, let us start with $\eta_i(t) \coloneqq \xi_{i}(t) \sqcup \xi_{i+d}(t)$.
	Then, since $y_i\xi_i(b)^{kn}\notin \dom(\xi_i(t))$ and $y_{i+d}\xi_{i+d}(b)^{mk+1} \notin \rng(\xi_{i+d}(t))$, we can set
	\[
	    y_i\xi_i(b)^{kn}\eta_i(t)\coloneqq y_{i+d}\xi_{i+d}(b)^{mk+1}
	\] 
	for all $k\in\Z$.
	We get in particular $y_i\eta_i(t)=y_{i+d}\xi_{i+d}(b)=y_{i+d}\eta_i(b)$.
	
	Afterwards, we extend $\eta_i$ to a transitive \( \BSo(m,n)  \)-action $\tilde \eta_i$, e.g.\ using the maximal forest saturation.
	Since $\tilde \eta_i$ extends $\xi_i$, the pointed action $[\tilde \eta_i,x_i]$ belongs to $U_i = \Nc([\alpha_i,x_i],R)$, 
	and since $\tilde \eta_i$ extends $\xi_{i+d}$, the pointed action $[\tilde \eta_i,x_{i+d}]$ belongs to $U_{i+d} = \Nc([\alpha_{i+d},x_{i+d}],R)$.
	By construction, 
	\[
	x_i\tilde\eta_i(g)=y_i
	\text{ and } y_i\tilde\eta_i(t)\tilde\eta_i(b)\inv=y_{i+d}\text{ and }x_{i+d}\tilde\eta_i(g)=y_{i+d}.
	\]
	It follows that by letting 
	\[
	\gamma \coloneqq g tb\inv g\inv
	\] 
	we have $x_i\tilde \eta_i(\gamma)=x_{i+d}$ for every $i\in\{1,\dots,d\}$.
	Since $[\tilde \eta_i,x_i]\in U_i$ and $[\tilde \eta_i,x_{i+d}]\in U_{i+d}$, this proves that for all $i\in\{1,\dots,d\}$, we have $U_i\gamma\cap U_{i+d}\neq\emptyset$ as wanted.

	\paragraph{Step 4 bis. In case of finite phenotype: U-turn.}
	We restart from the end of Step 3 and, from now on, we assume that the phenotype is finite.
	Let $y_i^0\coloneqq y_i\cdot \xi_i(b)\neq y_i$.
	Define inductively a sequence of labels $(l^j_i)_j$ as follows:  $l^0_i \coloneqq L(v_i)$, 
	which is also the cardinality of the $\xi_i(b)$-orbit  of $y_i^0$,
	and then for all $j\geq 0$, $l^{j+1}_i$ is the smallest positive integer satisfying the equation   
	\begin{equation}\label{eq: transfer yj - y j+1}
		\frac{l_i^{j+1}}{\gcd(l_i^{j+1},n)}=\frac{l_i^{j}}{\gcd(l_i^j,m)}.
	\end{equation}

	\begin{claim} \label{claim: lij is stationary}
		From a certain rank $r$, we have $l_i^j = q$ for all $i$ and for all $j \geq r$ 
		(i.e., we eventually reach the phenotype for all $i$).
	\end{claim}
	Let us finish our argument before proving the claim: fix $r$ as above.
	We extend $\xi_i$ by first adding a collection of $r$ extra $\xi_i(b)$-orbits $(O_i^j)_{1\leq j\leq r}$, with each $O_i^j$  of cardinality $l_i^j$.
	Then pick a point $y_i^j$ in each $O_i^j$.
	The transfer equation~\eqref{eq: transfer yj - y j+1} guarantees that
	the $\xi_i(b^n)$-orbit of $y_{i+1}$ and the $\xi_i(b^m)$-orbit of $y_i$
	have the same cardinality.
	Moreover, as $\abs{m} \geq 2$, observe that $y_i$ and $y_i^0$ belong to disjoint $\xi_i(b)^m$-orbits, so that $y_i^0$ is not in the target of $\xi_i(t)$.
	Therefore, we can extend $\xi_i(t)$ by letting for every $0 \leq j \leq r-1$ and $k\in\Z$:
	\[
	y_i^{j+1}\cdot \xi_i(t) \coloneqq y_i^j  \text{\ \  and \ \ } y_i^{j+1}\xi_i(b)^{nk} \cdot \xi_i(t) \coloneqq y_i^j\xi_i(b)^{mk}.
	\]

	We now have $2d$ transitive preactions $\xi_i$ which extend 
	the Schreier balls of radius $R$ centered at $x_i$ of our original actions $\alpha_i$, 
	and a large word $wbt^{-r}$, with $r\geq 0$, 
	such that for all $i\in\{1,\dots,2d\}$, by following the word $wbt^{-r}$ and starting at $x_i$, 
	we end up at the point $y^r_i$ whose \( b \)-orbit's cardinality is equal to the phenotype $q$. 

	\begin{cproofbis}{Proof of Claim \ref{claim: lij is stationary}}
		Our claim is equivalent to stating that, from a certain rank, $\abs{l_i^j}_p = 0$ for all $i$ and for every prime $p$ that does not divide the phenotype. Let us thus fix 
		such a prime $p$.
		
		First, the Transfer Equation~\eqref{eq:p adic transfert} implies:
		\begin{equation}\label{eq: def mngraphes with max for l j}
			\max(\abs{l_i^{j+1}}_p - \abs{n}_p,0)
			= \max(\abs{l_i^{j}}_p - \abs{m}_p, 0).
		\end{equation}
		We consider three cases, depending on the sign of $\vert m\vert _p-\vert n\vert_p$.
		
		\begin{itemize}
			\item If $\vert m\vert_p < \vert n\vert _p$, then Lemma \ref{lem: leaving the seed uniformly} (Condition \ref{item: separation in mfs}) 
			gives $\abs{l_i^0}_p = \abs{L(v_i)}_p = \abs{m}_p$, and Equation \eqref{eq: def mngraphes with max for l j} allows $\abs{l_i^{j+1}}_p = 0$ starting from $j = 0$ (which happens since we choose $l_i^{j+1}$ as small as possible).

			\item If $\vert m\vert _p>\vert n\vert_p$, then, as long as $\abs{l_i^j}_p > 0$, we have:
			
			\[
			\abs{l_i^{j+1}}_p = 
			\begin{cases}
				\abs{l_i^j}_p - \abs{m}_p + \abs{n}_p & \text{if} \, \abs{l_i^j}_p > \abs{m}_p; \\
				0 & \text{if} \, \abs{l_i^j}_p \leq \abs{m}_p.
			\end{cases}
			\]
			
			This follows from equation \eqref{eq: def mngraphes with max for l j}.
			Hence, from a certain rank $r_p$, we get $\abs{l_i^j}_p = 0$.
			
			\item If $\vert m\vert _p=\vert n\vert_p$, we have $\abs{l_i^0}_p \leq \abs{m}_p$ since $p$ is assumed not to divide the phenotype. As before, Equation~\eqref{eq: def mngraphes with max for l j} allows $\abs{l_i^j}_p = 0$ starting from $j = 1$.
			
		\end{itemize}
		Taking $r$ larger than the finitely many $r_p$'s arising when $\vert m\vert _p>\vert n\vert_p$, the claim is proven.
	\end{cproofbis}

	\paragraph{Step 5. Final welding.}
	
	For every $i\in\{1,...,d\}$, we weld $\xi_i$ and $\xi_{i+d}$ into a new transitive preaction $\eta_i$ as follows.
	The domain of $\eta_i$ is the disjoint union of the domain of $\xi_i$, the domain of $\xi_{i+d}$, 
	and a set $O_i$ of cardinal equal to $\frac{nq}{\gcd(n,q)}$, i.e.\ the least common multiple $n$ and $q$. 
	The bijection $\eta_i(b)$ is the disjoint union of $\xi_i(b)$, and $\xi_{i+d}(b)$, and of a cycle of length $\frac{nq}{\gcd(n,q)}$ on $O_i$.
	
	In order to define $\eta_i(t)$, let us start with $\eta_i(t) \coloneqq \xi_{i}(t) \sqcup \xi_{i+d}(t)$.
	Then, we fix some $z_i\in O_i$ and we want to extend $\eta_i(t)$ in such a way that $z_i$ is sent onto $y^r_i$ and $z_i\eta_i(b)$ is sent onto $y^r_{i+d}$.
	In order to enforce
	the equivariance condition in the definition of preactions, observe
	that the following orbits all have the same cardinal, namely $\frac{q}{\gcd(n,q)}=\frac{q}{\gcd(m,q)}$:  the $\eta_i(b)^n$-orbits of $z_i$ and of $z_i \eta_i(b)$, the $\eta_i(b)^m$-orbits of $y^r_i$ and of $y^r_{i+d}$. 
	Moreover, since the cardinality of the $\eta_i(b)$-orbit of $z_i$ is a multiple of $n$, the $\eta_i(b)^n$-orbits of $z_i$ and $z_i \eta_i(b)$ are disjoint.
	This allows us to extend $\eta_i(t)$ as wanted
	by setting
	\[
	\left( z_i \eta_i(b)^{nk} \right) \cdot \eta_i(t) \coloneqq y^r_i\eta_i(b)^{mk}
	\quad \text{ and } \quad
	\left( z_{i} \eta_i(b)^{nk+1} \right)\cdot  \eta_i(t) \coloneqq y^r_{i+d}  \eta_{i}(b)^{mk}
	\] 
	for all $k\in\Z$.
	Then, we extend $\eta_i$ to a transitive \( \BSo(m,n)  \)-action $\tilde \eta_i$, e.g.\ using the maximal forest saturation as described in Section \ref{Sect: forest saturations}.
	
	Since $\tilde \eta_i$ extends $\xi_i$, the pointed action $[\tilde \eta_i,x_i]$ belongs to $U_i$,  and since $\tilde \eta_i$ extends $\xi_{i+d}$, the pointed action $[\tilde \eta_i,x_{i+d}]$ belongs to $U_{i+d}$.
	By construction, 
	\[
	x_i\tilde\eta_i(wbt^{-(r+1)})=z_i\text{ and }x_{i+d}\tilde\eta_i(wbt^{-(r+1)})=z_i\tilde\eta_i(b).
	\]
	It follows that by letting 
	\[
	\gamma \coloneqq \left(wbt^{-(r+1)}\right) b \left(wbt^{-(r+1)}\right)\inv
	\] 
	we have $x_i\tilde \eta_i(\gamma)=x_{i+d}$ for every $i\in\{1,\dots,d\}$.
	Since $[\tilde \eta_i,x_i]\in U_i$ and $[\tilde \eta_i,x_{i+d}]\in U_{i+d}$, this proves that for all $i\in\{1,\dots,d\}$, we have $U_i\gamma\cap U_{i+d}\neq\emptyset$ as wanted.
\end{proof}

\section{High transitivity results} \label{Sec: HT}

In this section, we provide the proof for Theorem \ref{th-intro: HT generic in Ph= 1 and oo}. 
Section \ref{Sect: phenotypical obstruction} describes the pieces of the phenotypical partition where there are no highly transitive actions (Theorem \ref{thm: phenotype and non primitivity} and Proposition \ref{prop: no primitive inf phen n=m}).
Section \ref{Sect: generic HT} then shows the genericity
of high transitivity in the remaining pieces (Theorem \ref{th: HT generic in Ph= 1 and oo}), thus completing the proof of Theorem \ref{th-intro: HT generic in Ph= 1 and oo}.

\subsection{Phenotypical obstruction to high transitivity}\label{Sect: phenotypical obstruction}

Recall that an action is \defin{primitive} when it preserves no non-trivial equivalence 
relation, and note that \(2\)-transitive (in particular highly transitive) actions
are always primitive. The purpose of this section is to prove the following result.

\begin{theorem}\label{thm: phenotype and non primitivity}
	Let \( m,n\in\Z\smallsetminus\{0\} \), let \( q \) be an \( (m,n) \)-phenotype 
	such that \(q\neq 1\) and \(q\neq\infty\). 
	Then every transitive \(\BSo(m,n)\)-action of phenotype \(q\) on an infinite set 
	fails to be primitive. 
	In particular, there are no highly transitive  \(\BSo(m,n)\)-actions of phenotype \(q\).
\end{theorem}
	The proof goes by exhibiting a natural nontrivial equivalence relation 
	preserved by actions of phenotype \(q\not\in\{1,\infty\}\) which we will shortly introduce.
    We begin with a preliminary definition. 
    
    \begin{definition}
    Given a permutation $\sigma\in\Sym(X)$ and a finite $\sigma$-orbit $O$, 
	a $\sigma$-\defin{suborbit} of $O$ is a subset $O'\subseteq O$ which is a $\sigma^\ell$-orbit for some $\ell\in\N$. 
    \end{definition}
    Given a any (finite) cycle $\sigma\in\Sym(X)$ of length $k$ and any $r$ dividing $k$,
    the permutation $\sigma^{\frac{k}{r}}$ generates the unique subgroup $G_r$ 
    of order $r$ in $\la\sigma\ra$.
    The support $O$ of $\sigma$ admits a unique partition of $O$ into $\sigma$-suborbits of cardinality $r$,
    which is the partition into $G_r$-orbits.
    This is also the partition of $O$ into $\sigma^{\frac{k}{r}}$-orbits, and $\sigma$ acts by permuting the pieces of this partition.
    
    We now relate these remarks to our setup as follows.
    Given a transitive action $X\curvearrowleft^\alpha \BSo(m,n)$ such that  $\PHE(\alpha) < \infty$, we call \defin{reduced phenotype} the integer
    \begin{equation}\label{eq: def reduced phenotype}
        \PHEred(\alpha) \coloneqq \frac{\PHE(\alpha)}{\gcd(\PHE(\alpha), m)} 
        = \frac{\PHE(\alpha)}{\gcd(\PHE(\alpha), n)}
    \end{equation}
    and notice that $\PHEred(\alpha) > 1$ whenever $\PHE(\alpha) > 1$ by definition of the phenotype.
    Every orbit $O$ of $\alpha(b)\in\Sym(X)$ has cardinality divisible by $\PHE(\alpha)$, hence also by $\PHEred(\alpha)$.
	\begin{remark}
		Observe that when \( m \) and \( n \) are coprime,
		the phenotype is always coprime with both \( m \) and \( n \), so 
		in this case the reduced phenotype coincides with the phenotype. 
		More generally, the reduced version alters the $p$-adic valuation of the phenotype
		only for those \(p\)'s such that 
		\[
		0< \abs m_p=\abs n_p <\abs{\PHE(\alpha)}_p, 
		\]
		in which case it becomes 
		\( 
		\abs{\PHEred(\alpha)}_p = \abs{\PHE(\alpha)}_p-\abs{m}_p
								= \abs{\PHE(\alpha)}_p-\abs{n}_p
		\).
	\end{remark}
	\begin{definition}
        Given a transitive action \( \alpha \) of finite phenotype on a set \( X \),
        we call the \defin{reduced \( b \)-orbit relation} of \( \alpha \) the equivalence relation $\Rc^{\mathrm{red}}_{\alpha(b)}$ on \( X \)
        whose classes are the $\alpha(b)$-suborbits of cardinality $\PHEred(\alpha)$.
    \end{definition}
    \begin{proposition}\label{Prop: phenotypical relation is invariant}
        For every transitive action $X\curvearrowleft^\alpha \BSo(m,n)$ such that $\PHE(\alpha) < \infty$,
        the reduced \(b\)-orbit relation $\Rc^{\mathrm{red}}_{\alpha(b)}$ is invariant under \( \alpha \). 
    \end{proposition}
    \begin{proof}
        First, $\Rc^{\mathrm{red}}_{\alpha(b)}$ is invariant under $\alpha(b)$,
        since inside every $\alpha(b)$-orbit, the suborbits of cardinality $r\coloneqq \PHEred(\alpha)$ are permuted by $\alpha(b)$.

       It remains to prove that $\alpha(t)$ also permutes the $\alpha(b)$-suborbits of cardinal $r$.
        Let us fix $x\in X$.
        Denote by $O$ the $\alpha(b)$-orbit of \( x \) and by $O'$ the $\alpha(b)$-suborbit of \( x \) of cardinality $r$.
        Its image $O' \cdot \alpha(t)$ clearly has the same cardinal $r$, so we only have to show $O' \cdot \alpha(t)$ is still an $\alpha(b)$-suborbit. 
        Let $\sigma$ denote the restriction of $\alpha(b)$ to $O$; this is a transitive permutation of $O$ of order $k\coloneqq\abs O$
        and the orbit of \( x \) under $\sigma' \coloneqq \sigma^{\frac{k}{r}}$ is equal to $O'$.

Let $q\coloneqq\PHE(\alpha)$. By Equation \eqref{eq: def reduced phenotype}, we have 
$\gcd(q,n)=\frac{q}{r}$. Bézout's identity provides some $u, v \in \Z$ such that $\gcd(q, n)=uq + vn.$
        So if $s\in\N$ satisfies  $k=sq$, 
        we obtain
        \[
            \frac{k}{r}
            = s\cdot\frac{ q}{r}
            = s \cdot \gcd(q, n)
            = s (uq + vn) .
        \]
        Then we get
        \[
            \sigma' = \sigma^{\frac{k}{r}}
            = \sigma^{suq + svn} 
            = \sigma^{u k + svn}
            = \sigma^{svn}.
        \]
        Hence, one has $O' = x \cdot \la \sigma' \ra = x \cdot \la \sigma^{svn} \ra=x \cdot \la \alpha(b)^{nsv} \ra$.
        Thus, using the defining relation $tb^m = b^nt$ in \( \BSo(m,n)  \), we get
         \[
            O' \cdot \alpha(t)
            = \big( x \cdot \la \alpha(b)^{nsv} \ra \big) \cdot \alpha(t)
            =\big( x \cdot \alpha(t) \big) \cdot \la \alpha(b)^{msv} \ra .
        \]
        The set $O'\cdot \alpha(t)$ is therefore an $\alpha(b)$-suborbit as desired.
    \end{proof}
    
	We can now easily prove the main result of this section. 

	\begin{proof}[Proof of Theorem \ref{thm: phenotype and non primitivity}]
		Let \( \alpha \) be a \(\BSo(m,n)\)-action of phenotype \( q\notin\{1,\infty\} \)
		on an infinite set \(X\).
		The phenotypical relation \(\Rc^{\mathrm{red}}_{\alpha(b)}\) is non-trivial: 
		all its equivalence classes have cardinal \(\PHEred(\alpha)>1\) as soon as \(\PHE(\alpha)>1\) 
		hence they are neither singletons nor equal to \(X\) which is infinite. 
	    Finally, Proposition \ref{Prop: phenotypical relation is invariant} ensures us that
		\(\Rc^{\mathrm{red}}_{\alpha(b)}\) is \( \alpha \)-invariant, 
		thus witnessing the non-primitivity of \( \alpha \).
	\end{proof}

	\begin{remark}
		Let \( \alpha \) be a transitive \(\BSo(m,n)\)-action of phenotype \(q\notin\{1,\infty\}\)
		on a set \(X\).
		Since the equivalence relation \(\Rc^{\mathrm{red}}_{\alpha(b)}\) is \( \alpha \)-invariant,
		we have a quotient \(\BSo(m,n)\)-action \(\tilde\alpha\) on the quotient set \(X/\Rc^{\mathrm{red}}_{\alpha(b)}\).
		By construction, given an \(\alpha(b)\)-orbit of cardinal \(k\),
		the corresponding \(\tilde\alpha(b)\)-orbit has cardinal \(l\), where \(l\) satisfies that
		for all prime \(p\) that if \(0<\abs{m}_p=\abs{n}_p<\abs{\PHE(\alpha)}_p \), then
		\begin{align*}
		\abs{l}_p & = \abs{k}_p-\abs{\PHEred(\alpha)}_p \\
				  & = \abs{\PHE(\alpha)}_p- (\abs{\PHE(\alpha)}_p-\abs{m}_p) \\
				  & = \abs{m}_p=\abs{n}_p,
		\end{align*}
		while \(\abs{l}_p = \abs{k}_p\) otherwise. 
		It follows that the quotient action \(\tilde\alpha\) has phenotype \(1\). 
	\end{remark}

    We conclude this section by examining the case of infinite 
    phenotype when \(\abs m=\abs n\), where a similar and 
    easier argument holds.
    
    \begin{proposition}\label{prop: no primitive inf phen n=m}
        Let \(\abs m=\abs n \geq 2\), then \(\BSo(m,n)\) has 
        no primitive transitive action of infinite phenotype.
    \end{proposition}
    \begin{proof}
        Let \(\alpha\) be a transitive \(\BSo(m,n)\)-action of infinite
        phenotype. By definition \(b\) acts freely, so the partition into \( b^m \)-orbits is non trivial. Since \( \la b^m\ra \) is a normal subgroup, this partition
        is invariant, and \(\alpha\) is thus not primitive.
    \end{proof}

\subsection{Genericity of highly transitive actions}\label{Sect: generic HT}
In this section, we prove the following:

\begin{theorem}\label{th: HT generic in Ph= 1 and oo}
    Let $\abs{m}\neq 1$ and $\abs{n}\neq 1$ and let $q$ be a \( (m,n) \)-phenotype.
    Let $\Gamma=\BSo(m,n)$ and let $\PK_q $ be the subset of the perfect kernel $\PK(\Gamma)$ consisting of actions of phenotype~$q$.
    The set of highly transitive actions is dense \( G_\delta \)  in $\PK_q $ when 
    \begin{itemize}
        \item either $q=1$,  
        \item or $q=\infty$ and  $\abs{m} \neq \abs{n}$.
    \end{itemize}
\end{theorem}

\begin{proof}[Proof of Theorem~\ref{th: HT generic in Ph= 1 and oo}]
	
	We decompose this proof into three parts: the first part does not depend on the phenotype,
	and then the proof splits between the case $q=1$ and the case $q=\infty$.

		\paragraph{Common part.}
	
	First recall that \(\mathcal K_q(\Gamma)\) is always Polish by \cite[Theorem~B]{CGLMS-22}
	and the fact that open and closed subspaces of Polish spaces are Polish for the induced topology.
	We thus use the characterization from Lemma~\ref{lem: HT and Baire} and show 
	that  Condition~(\textasteriskcentered) of Lemma~\ref{lem: HT and Baire} therefrom holds for \(\mathcal P=\PK_q \).
	
	\bigskip
	Let us thus fix some  $[\alpha,x]\in\PK_q $ and $g_1,\dots,g_{2d}\in\Gamma$
	such that the elements \(x\alpha(g_1),\dots,x\alpha(g_{2d})\) are pairwise distinct.
	Let \(\mathcal V\) be a neighborhood of $[\alpha,x]$. We must exhibit a pointed transitive action
	\([\alpha',x']\in \mathcal V\cap \PK_q \) and \(\gamma\in\Gamma\) such that for all 
	\(i\in\{1,\dots,2d\}\), we have 
	\[
	x'\alpha'(g_i) \alpha'(\gamma) = x'\alpha'(g_{i+d})
	\]

    Shrinking $\mathcal V$ if necessary, we assume that $\mathcal V=\mathcal N([\alpha,x],R)$ for some $R>0$.
    It suffices to start by setting $x'\coloneqq x$ and to construct a transitive
    action \(\alpha'\) whose Schreier \(R\)-ball at the basepoint $x$ coincides with that of \( \alpha \).
    Furthermore, we may as well assume that $R$ is larger than the length of each $g_i$ 
    with respect to the generating set $\{t,b\}$. Thus, for \(i\in\{1,\dots,2d\}\), the points \(x_i \coloneqq x\alpha(g_i)\) belong to the ball \( B_\alpha(x,R)\).
    Our aim is now to find \([\alpha',x]\in\PK_q \) 
    and \(\gamma\in\Gamma\) such that 
    \(B_{\alpha'}(x,R)=B_\alpha(x,R)\)
    and for all \(i\in\{1,\dots,d\}\), \[x_i\alpha'(\gamma)=x_{i+d}.\]	 
    
    Let \(B\) be the union of the \(b\)-orbits that intersect the ball $B_\alpha(x,R+1)$. 
	 Since $[\alpha,x]\in\mathcal K(\Gamma)$, the action $\alpha$ has infinitely many $b$-orbits, so the set \(B\) is a proper subset of the domain of $\alpha$.

	 Consider the transitive pre-action $\alpha^0$ obtained by restricting \( \alpha \) to \(B\) .
	 We now replace \( \alpha \) by the maximal forest saturation of $\alpha^0$. 
	 This does not affect the definition of $\mathcal V=\mathcal N([\alpha,R])$ by construction, 
	 nor the definition of each $x_i=x\alpha(g_i)$ since $R$ was taken larger than the length of each $g_i$.

	 Let $X$ be the domain of $\alpha$.
	 In order to apply Lemma \ref{lem: leaving the seed uniformly}, let us consider the family of pointed preactions \((\alpha^0,x_1), \dots,(\alpha^0,x_{2d})\), and the associated
	 maximal forest saturations \(\alpha_1\coloneqq\alpha,\dots,\alpha_{2d} \coloneqq \alpha\).
	 For $i\in\{1,\dots,2d\}$,  let
	$\pi_i:\Tree\to \BSe(\alpha)$ be the graph morphism associated 
	to the unique \( \Gamma \)-equivariant map $\hat\pi_i:\Gamma\to X$ taking $\id$ to $x_i$. We now take the path \( c \) provided by Lemma~\ref{lem: leaving the seed uniformly}.

	Let $r$ be a geodesic segment in $\Tree$ of length larger than that of $c$, such that $\source(r) = \target(c)$, and consisting only of positive edges.
	The path $cr$ still satisfies the conclusions of Lemma \ref{lem: leaving the seed uniformly} (by its ``Moreover'' part).

	It follows that for every \(i,j\), if the target
	\(\pi_j(\target(cr))\) belongs to $\pi_i(cr)$, then, in fact, $\pi_j(\target(cr))$ belongs to $\pi_i(r)$ and not to $\pi_i(c)$.
	Thus, the unique path from \(\pi_j(\target(cr))\) to $\pi_i(\target(cr))$ is a subpath of $\pi_i(r)$), thus made of positive edges.
	We now split the proof in two cases.

	\paragraph{Case \(q=\infty\) and $\abs{m} \neq \abs{n}$.}
	Since \(\BSo(m,n)\) and $\BSo(n,m)$ are isomorphic, we will treat only the case \(\abs{m} < \abs{n}\). 
	Let us then fix a prime \(p\) such that $\abs{m}_p < \abs{n}_p$. 
	
	Let \(g\in \Gamma\) be such that $g\la b \ra = \target(cr)$.
	By construction, for every \(i\in\{1,\dots,2d\}\) we have 
	\(x_i\alpha(g) \la b \ra = \pi_i(\target(cr))\) in the Bass-Serre graph.

	\begin{claim} 
	For large enough \(k\in\N\) the images 
	$x_1\alpha(gt^k),\dots,x_{2d}\alpha(gt^k)$ all belong to different \( b \)-orbits.
	\end{claim}
	\begin{cproof}
	It suffices to show, for any two indices \(i\neq j\in\{1,\dots,2d\}\),
	that for large enough 
	\(k\), the points \(x_i\alpha(gt^k)\) and \(x_j\alpha(gt^k)\) belong to distinct \(b\)-orbits.
	So let us fix $i\neq j$.

	First note that 
	if for some \(k\geq 1\) the two elements \(x_i\alpha(gt^k)\), \(x_j\alpha(gt^k)\) 
	belong to the same \(b\)-orbit, they correspond
	 to the same vertex in the Bass-Serre graph, 
	 and hence must come 
	 from the same positive edge since they belong to the
	 forest part. 
	 By construction of the maximal forest saturation,
	it follows that \(x_i\alpha(gt^k)\) and \(x_j\alpha(gt^k)\)   belong to the same \(b^m\)-orbit, and their predecessors
	\(x_i\alpha(gt^{k-1})\), \(x_j\alpha(gt^{k-1})\) belong to the same \(b^n\)-orbit.
	It follows that \(x_i\alpha(g)\), \(x_j\alpha(g)\) belong to the same \(b\)-orbit.

	Now assume we have $K\in\Z$ such that $x_i\alpha(g)\alpha(b^K)=x_j\alpha(g)$ 
	and \(L\in\Z\) such that $x_i\alpha(gt)\alpha(b^L)=x_j\alpha(gt)$ 
	(witnessing that their \( t \)-images are also in the same \( b \)-orbit).
	Then by the observation we just made we must have \(n\) divides \(K\).
	But then using the Baumslag-Solitar relation \(tb^mt\inv =b^n\) and \(\abs{m}_p<\abs{n}_p\), 
	we obtain that
    $\abs{L}_p<\abs{K}_p$. Iterating this argument, we conclude that once \(k>\abs{K}_p\), 
	the elements \(x_i\alpha(gt^k)\) and \(x_j\alpha(gt^k)\) must belong to distinct \(b\)-orbits. 
	This finishes the proof of the claim since there are only finitely many pairs of indices \((i,j)\) to consider.
	\end{cproof}

    We now fix \(k\) large enough so that the points
    $y_1 \coloneqq x_1\alpha(gt^k),\dots, y_{2d} \coloneqq x_{2d}\alpha(gt^k)$ belong to pairwise disjoint \( b \)-orbits.

	Let \(\mathcal G\) be the connected subgraph of \(\BSe(\alpha)\)
	obtained as the union of \(\BSe(\alpha_0)\) with all the paths 
	\(\pi_i(cr)\)  and the positive edges \(x_i\alpha(gt^l)\alpha(\la b^n\ra)\) 
	extending them for 
	\(l\in\{0,\dots,k-1\}\). 
	Note that 
	for every \(i\in\{1,\dots,2d\}\), the vertex corresponding to the \( b \)-orbit 
	of $y_i$ 
	is in the complement of 
	$\BSe(\alpha^0)$. Moreover, this vertex has
	ingoing degree $m$ in \( \BSe(\alpha) \), 
    while it has ingoing degree $1$ in $\mathcal G$
	by the observation we made right before 
	splitting the proof in two cases.
	
    Let $G$ be the subgraph of the Schreier graph of \( \alpha \) consisting of the 
	\( b \)-orbits that are shrunk to vertices of $\mathcal G$ 
    and all the edges that are shrunk to edges of $\mathcal{G}$. 
    It is the Schreier graph of a pre-action that we denote \( \xi \). 
	
	We now extend back \(\xi\) as follows.
	We first add \(d\) new infinite \(\xi(b)\)-orbits \(O_1,\dots,O_d\), and pick for each \(j\in\{1,\dots,d\}\)
	one point \(z_j\in O_j\). Observe that \(z_j\) and \(z_j\xi(b)\) belong to distinct \(b^m\)-orbits.
	We can then extend \(\xi\) further by letting, for every \(l\in\Z\): 
	\[
	z_j \xi(b^{lm})\xi(t)=y_j\xi(b) \xi(b^{ln}) =\text{ and } 
	z_j \xi(b^{lm+1})\xi(t)=y_{d+j} \xi(b).
	\]
	Finally let \(\alpha'\) be an arbitrary forest-saturation of \(\xi\), 
	we claim that $\alpha'$ 
    is the transitive action we are after.
    Indeed, since we did not modify the path $\pi_i(cr)$ nor
	the positive edges \(x_i\alpha(gt^l)\alpha(\la b^n\ra)\) following it for 
	\(l\in\{1,\dots,k\}\), 
	we have $x_i\alpha(gt^k) = x_i\alpha'(gt^k) = y_i$ 
	for all $i\in\{1,\dots,2d\}$.
    By construction, we now have, letting \(\gamma=(gt^kbt\inv)b(gt^kbt\inv)^{-1}\), that:
    \begin{align*}
		x_j\alpha'(\gamma)&=x_j\alpha'(gt^k bt^{-1}btb\inv t ^{-k}g^{-1})\\
         & = y_j\alpha'(bt^{-1}btb^{-1}t^{-k}g^{-1})\\
        & = z_j\alpha'(btb^{-1}t^{-k}g^{-1})\\
        & = y_{j+d}\alpha'(t^{-k}g\inv)\\
        & = x_{j+d}.
    \end{align*}
    Moreover, since $\alpha'$ extends \( \alpha_0 \), 
	it belongs to \( \mathcal V=\mathcal N(\alpha,R) \) as wanted.

	\paragraph{Case $q=1$.}
    Let $\mathcal G$ be the subgraph of \( \BSe(\alpha) \)  consisting of the union of $\BSe(\alpha^0)$ with all the $\pi_i(cr)$, for $i=1, \cdots, 2d$.
   	As in the previous case, each $\pi_i(\target(cr))$ is in the complement of $\BSe(\alpha^0)$. Since it comes from a positive edge, its label satisfies $\gcd(L(\pi_i(\target(cr))),m)=\abs{m}$ by the transfer rule. 
   	So it has ingoing degree $\abs{m}$ in \( \BSe(\alpha) \), 
    while it has ingoing degree $1$ in $\mathcal G$.
	
    Let $G$ be the subgraph of the Schreier graph of \( \alpha \) consisting of the \( b \)-orbits that are shrunk to vertices of $\mathcal G$ 
    and all the edges that are shrunk to edges of $\mathcal{G}$. 
    As in the first case, $G$ is the Schreier graph of a pre-action \( \xi \) which 
    extends \( \alpha_0 \) and whose Bass-Serre graph is $\mathcal G$.
    
    We now replace \( \alpha \) by the minimal forest-saturation of \( \xi \), which we still denote \( \alpha \). 
    Then \( \alpha \) still extends \( \xi \) and thus it also extends $\alpha^0$. 
    In particular, this new modification of \( \alpha \) does not modify the open set $\mathcal N([\alpha,R])$ nor the elements $x_i=x\alpha(g_i)$. 
	Moreover, for every $i\in\{1,\dots,2d\}$ the path $\pi_i(cr)$ is also left unchanged since \( \alpha \) extends \( \xi \).
	Its terminal vertex has the same label as before and hence has the same (maximal) ingoing degree as before, namely $m$.

    We now extend the path $cr$ in $\Tree$ to a reduced (infinite) path $crs$ by adding a geodesic ray only made  of negative edges. 
    Since both $\target (cr)$ and $\pi_i(\target(cr))$ have $m$ ingoing edges, the path $\pi_i(crs)$ is reduced at the vertex $\pi_i(\target(cr))$. Moreover, the orientation of the two edges incident to $\pi_i(\target(cr))$ in $\pi_i(crs)$ being opposite, while the ongoing degree of $\pi_i(\target(cr))$ in $\mathcal G$ is $1$, it follows that the path $\pi_i(s)$ 
    is a geodesic whose intersection with $\Gc$ is reduced to its source $\source(\pi_i(s))$.

    Let us denote by $(v_{i,k})_{k\geq 0}$ the vertices of $\pi_i(s)$ in \( \BSe(\alpha) \).
    The following version of Claim \ref{claim: lij is stationary} holds.
    \begin{claim}\label{claim: arriving at labels 1}
        For $k$ large enough, the labels satisfy $L(v_{i,k}) = 1$ for $ i=1,\ldots, 2d$. In other words, the \( b \)-orbits corresponding to $v_{i,k}$ are singletons.
    \end{claim}
    \begin{cproof}
        Let us fix some index $i$. 
        It suffices to prove $L(v_{i,k}) = 1$ for $k$ large enough.
        So let us abbreviate $v_{i,k}$ as $v_k$.
        For each $k$, let $e_k$ be the positive edge from $v_{k+1}$ to $v_k$.
        The Transfer Equation~\eqref{eq:p adic transfert} implies:
        \[
	\max(\abs{L(v_{k+1})}_p - \abs{n}_p,0)
	= \abs{L(e_k)}_p
	= \max(\abs{L(v_k)}_p - \abs{m}_p, 0).
\]
    Since we are in the minimal forest saturation, we then have
    \begin{equation}\label{eq: transfer min-forest sat}
        \abs{L(v_{k+1})}_p = \begin{cases}
            0 & \text{ if } \abs{L(v_k)}_p \leq \abs{m}_p \\
            \abs{L(v_k)}_p - \abs{m}_p + \abs{n}_p & \text{ if } \abs{L(v_k)}_p > \abs{m}_p .
        \end{cases}
    \end{equation}
    Let us now fix a prime $p$. We have two cases to consider:
    \begin{enumerate}
        \item  If $\abs{m}_p \leq \abs{n}_p$, 
        using that $cr$ satisfies the conclusions of Lemma \ref{lem: leaving the seed uniformly}, 
        that $v_0 = \target(\pi_i(cr))$ and that the phenotype is $1$, we have: $\abs{L(v_0)}_p = \abs{m}_p$.
        The transfer equation 
        \eqref{eq: transfer min-forest sat} (first case) gives $\abs{L(v_k)}_p = 0$ for every $k\geq 1$.
        
        \item If $\abs{m}_p >\abs{n}_p$,
        we prove that $\abs{L(v_k)}_p = 0$ for sufficiently large $k$.
        Indeed, the transfer equation \eqref{eq: transfer min-forest sat} (second case) forces $\abs{L(v_{k + 1}))}_p < \abs{L(v_k)}_p$  as long as 
        $\abs{L(v_k)}_p > \abs{m}_p$. 
        Thus, we get $\abs{L(v_{k_0})}_p \leq \abs{m}_p$ for some index $k_0$
        and then $\abs{L(v_k)}_p=0$  for all $k \geq k_0 + 1$ by Equation \eqref{eq: transfer min-forest sat} (first case).
    \end{enumerate}
    Since there are only finitely many primes such that $\abs{m}_p >\abs{n}_p$, we conclude $L(v_k) = 1$ for $k$ large enough.
    \end{cproof}
    
    By the above Claim~\ref{claim: arriving at labels 1}, we can pick an initial segment $s'$ of $s$ such that $\pi_i(\target(s'))$ has label $1$ for every $i\in\{1, \ldots, 2d\}$.
    For $i\in\{1,\dots,2d\}$, let $y_i$ be the unique element of the label $1$ vertex $\pi_i(\target(s'))$, viewed as a \( b \)-orbit.

    Fix $g\in \BSo(m,n)$ such that $g\la b\ra = \target(crs')$ in the Bass-Serre $\Tree$.
    For $i\in\{1,\dots,2d\}$, we have by equivariance and the definition of $\pi_i$ that $y_i=x_i\alpha(g)$.
    Since $\alpha(g)$ is a bijection, we conclude that the $y_i$'s are pairwise distinct.
    In particular, their image in the Bass-Serre graph, i.e.\ the vertices  $\pi_i(\target(s'))$, are pairwise distinct.
    This is the point where it is crucial that the label is $1$, and thus to be in phenotype $1$.

    We now carry out one last modification of \( \alpha \): let $\mathcal H$ be the subgraph of \( \BSe(\alpha) \)  obtained as the union of $\BSe(\alpha^0)$ with all the paths $\pi_i(crs')$, for $i=1, \ldots, 2d$.
    In fact, $\Hc$ is the union of $\Gc$ and the geodesics $\pi_i(s')$, for $i=1, \ldots, 2d$.
    
     \medskip    
    \begin{center}
    \includegraphics[scale=1]{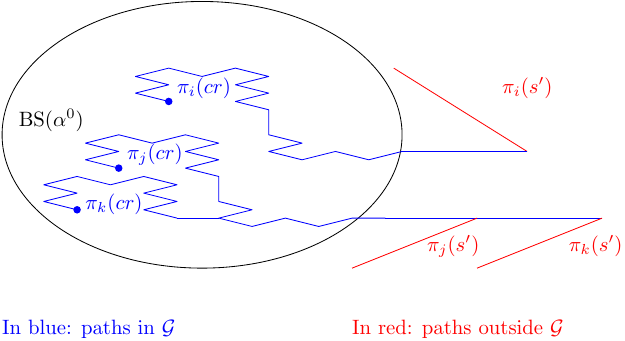}
    \end{center}
    
    Since $\pi_i(s)$ is a geodesic whose intersection with $\Gc$ is reduced to its source $\source(\pi_i(s))$, all $\pi_i(\target(crs'))$ are at the same distance from $\Gc$ in $\Hc$, namely the length of $s'$, and therefore they all have degree $1$ in $\Hc$. 
    In particular, for every $i$, the (unique) positive edge arriving at $\pi_i(\target(s'))=\{y_i\}$ lies outside $\Hc$.

    As before, the subgraph of the Schreier graph obtained by pulling back $\Hc$ is the Schreier graph of a pre-action that we denote $\eta$. 
    Observe that $\eta$ extends \( \xi \), hence also extends $\alpha^0$; its Bass-Serre graph is $\Hc$. 
       
    We now extend $\eta$ as follows:
    we add \( d \) new $\eta(b)$-orbits $O_j$ of cardinal $m$, for $j\in\{1,\dots,d\}$, and we pick one point $z_j$ in $O_j$.
    We connect $O_j$ to both $y_j$ and $y_{j+d}$ by declaring $z_j\eta(t)=y_j$
    and $z_j\eta(b)\eta(t)=y_{j+d}$.
    This is a genuine pre-action since $O_j$ decomposes as $m$ $\eta(b^m)$-orbits of cardinality $1$.
    Let $\alpha'$ be a forest-saturation of $\eta$; let us show that \(\alpha'\)
	is the action we seek.
	First, since we did not modify the path $\pi_i(crs')$, 
	we have $x_i\alpha(g) = x_i\alpha'(g) = y_i$ 
	for all $i\in\{1,\dots,2d\}$.
    By construction, we now have
    \begin{align*}
        x_j\alpha'(g t^{-1}btg^{-1}) & = y_j\alpha'(t^{-1}btg^{-1})\\
        & = z_j\alpha'(btg^{-1})\\
        & = y_{j+d}\alpha'(g\inv)\\
        & = x_{j+d}.
    \end{align*}
    Moreover, since $\alpha'$ extends \( \alpha_0 \), 
	it belongs to \( \mathcal V=\mathcal N(\alpha,R) \) as wanted. 
	This finishes the proof of the second and last case,
	so Theorem~\ref{th: HT generic in Ph= 1 and oo} is proved.
\end{proof}

\bibliographystyle{alphaurl}

\begin{thebibliography}{FLMMS22}

\bibitem[AG25]{azuelosPerfectKernelDynamics2023}
P\'{e}n\'{e}lope Azuelos and Damien Gaboriau.
\newblock Perfect kernel and dynamics: from {B}ass-{S}erre theory to hyperbolic
  groups.
\newblock {\em Math. Ann.}, 391(3):4733--4789, 2025.
\newblock \href {https://doi.org/10.1007/s00208-024-03038-w}
  {\path{doi:10.1007/s00208-024-03038-w}}.

\bibitem[Bon24]{Bontemps-pkGBS}
Sasha Bontemps.
\newblock Perfect kernel of generalized {{Baumslag-Solitar}} groups.
\newblock {\em arXiv preprint}, 2024.
\newblock \href {https://doi.org/10.48550/arXiv.2411.03221}
  {\path{doi:10.48550/arXiv.2411.03221}}.

\bibitem[CGLMS25]{CGLMS-22}
Alessandro Carderi, Damien Gaboriau, Fran{\c c}ois Le~Ma{\^i}tre, and Yves
  Stalder.
\newblock On the space of subgroups of {{Baumslag}}--{{Solitar}} groups {{I}}:
  {{Perfect}} kernel and phenotype.
\newblock {\em Revista Matem\'atica Iberoamericana}, 41(5):1711--1758, 2025.
\newblock \href {https://doi.org/10.4171/rmi/1549}
  {\path{doi:10.4171/rmi/1549}}.

\bibitem[Cha12]{chaynikov_properties_2012}
Vladimir~Vladimirovich Chaynikov.
\newblock {\em Properties of hyperbolic groups: free normal subgroups,
  quasiconvex subgroups and actions of maximal growth}.
\newblock PhD thesis, Vanderbuilt University, Nashville, June 2012.
\newblock URL:
  \url{https://ir.vanderbilt.edu/items/85ce827f-2af2-4c23-9713-7d9c7564ea53}.

\bibitem[Dix90]{dixon_most_1990}
John Dixon.
\newblock Most finitely generated permutation groups are free.
\newblock {\em Bulletin of the London Mathematical Society}, 22:222--226, 1990.
\newblock \href {https://doi.org/10.1112/blms/22.3.222}
  {\path{doi:10.1112/blms/22.3.222}}.

\bibitem[FLMMS22]{fima_characterization_HT_2022}
Pierre Fima, François Le~Maître, Soyoung Moon, and Yves Stalder.
\newblock A characterization of high transitivity for groups acting on trees.
\newblock {\em Discrete Analysis}, 2022(8):63 pp, August 2022.
\newblock \href {https://doi.org/10.19086/da.37645}
  {\path{doi:10.19086/da.37645}}.

\bibitem[FMS15]{fima_highly_2015}
Pierre Fima, Soyoung Moon, and Yves Stalder.
\newblock Highly transitive actions of groups acting on trees.
\newblock {\em Proceedings of the American Mathematical Society},
  143(12):5083--5095, December 2015.
\newblock \href {https://doi.org/10.1090/proc/12659}
  {\path{doi:10.1090/proc/12659}}.

\bibitem[GG13]{garionHighlyTransitiveActions2013}
Shelly Garion and Yair Glasner.
\newblock Highly transitive actions of $\operatorname{Out}({F}_n)$.
\newblock {\em Groups, Geometry, and Dynamics}, 7(2):357--376, 2013.
\newblock \href {https://doi.org/10.4171/ggd/185} {\path{doi:10.4171/ggd/185}}.

\bibitem[HO16]{hullTransitivitydegreescountable2016}
Michael Hull and Denis Osin.
\newblock Transitivity degrees of countable groups and acylindrical
  hyperbolicity.
\newblock {\em Israel Journal of Mathematics}, 216(1):307--353, 2016.
\newblock \href {https://doi.org/10.1007/s11856-016-1411-9}
  {\path{doi:10.1007/s11856-016-1411-9}}.

\bibitem[Kec95]{kechris_classical_1995}
Alexander~S. Kechris.
\newblock {\em Classical {Descriptive} {Set} {Theory}}, volume 156 of {\em
  Graduate {Texts} in {Mathematics}}.
\newblock Springer New York, New York, NY, 1995.
\newblock \href {https://doi.org/10.1007/978-1-4612-4190-4}
  {\path{doi:10.1007/978-1-4612-4190-4}}.

\bibitem[Kit12]{kitroserHighlytransitiveActionsSurface2012}
Daniel Kitroser.
\newblock Highly-transitive actions of surface groups.
\newblock {\em Proceedings of the American Mathematical Society},
  140(10):3365--3375, 2012.
\newblock \href {https://doi.org/10.1090/S0002-9939-2012-11195-5}
  {\path{doi:10.1090/S0002-9939-2012-11195-5}}.

\bibitem[LBMB22]{leboudecConfinedSubgroupsHigh2022}
Adrien Le~Boudec and Nicol{\'a}s Matte~Bon.
\newblock Confined subgroups and high transitivity.
\newblock {\em Annales Henri Lebesgue}, 5:491--522, 2022.
\newblock \href {https://doi.org/10.5802/ahl.128} {\path{doi:10.5802/ahl.128}}.

\bibitem[LM24]{le_maitre_polish_2024}
François Le~Maître.
\newblock {\em Polish groups and non free actions in the discrete or measurable
  context}.
\newblock {HDR}, Université Paris Cité, March 2024.
\newblock URL: \url{https://theses.hal.science/tel-05335793}.

\bibitem[LS01]{lyndon_combinatorial_2001}
Roger~C. Lyndon and Paul~E. Schupp.
\newblock {\em Combinatorial {Group} {Theory}}, volume~89 of {\em Classics in
  {Mathematics}}.
\newblock Springer Berlin Heidelberg, Berlin, Heidelberg, 2001.
\newblock \href {https://doi.org/10.1007/978-3-642-61896-3}
  {\path{doi:10.1007/978-3-642-61896-3}}.

\bibitem[McD77]{mcdonoughPermutationRepresentationFree1977}
T.~P. McDonough.
\newblock A permutation representation of a free group.
\newblock {\em The Quarterly Journal of Mathematics}, 28(3):353--356, 1977.
\newblock \href {https://doi.org/10.1093/qmath/28.3.353}
  {\path{doi:10.1093/qmath/28.3.353}}.

\bibitem[MS13]{moonHighlyTransitiveActions2013}
Soyoung Moon and Yves Stalder.
\newblock Highly transitive actions of free products.
\newblock {\em Algebraic \& Geometric Topology}, 13(1):589--607, 2013.
\newblock \href {https://doi.org/10.2140/agt.2013.13.589}
  {\path{doi:10.2140/agt.2013.13.589}}.

\bibitem[Ser80]{serre_Trees_1980}
Jean-Pierre Serre.
\newblock {\em Trees}.
\newblock Springer, Berlin, Heidelberg, 1980.
\newblock \href {https://doi.org/10.1007/978-3-642-61856-7}
  {\path{doi:10.1007/978-3-642-61856-7}}.

\end{thebibliography}

\bigskip
{\footnotesize
	
	\noindent
	{D.~G., \textsc{CNRS, ENS-Lyon, 
			Unité de Mathématiques Pures et Appliquées,  69007 Lyon, France}}
	\par\nopagebreak \texttt{damien.gaboriau@ens-lyon.fr}
	
	\medskip

	\noindent
	{F.~L.M., \textsc{Université Bourgogne Europe, CNRS, IMB UMR 5584, 21000 Dijon, France}}
	\par\nopagebreak \texttt{flemaitre@math.cnrs.fr}
	
	\medskip
	
	\noindent
	{Y.~S., \textsc{Université Clermont Auvergne, CNRS, LMBP, F-63000 Clermont–Ferrand, France}}
	\par\nopagebreak \texttt{yves.stalder@uca.fr}
}

\end{document}